\documentclass[11pt]{article}
\usepackage{amsfonts,amsmath, amssymb,latexsym,mathrsfs,rotating}
\usepackage[all,cmtip]{xy}
\def\Z{{\mathbb Z}}

\def\SL{{\rm SL}}
\def\GL{{\rm GL}}
\def\Gal{{\rm Gal}}
\def\PGL{{\rm PGL}}

\def\Sym{{\rm Sym}}

\def\SO{{\rm SO}}
\def\P{{\mathbb P}}
\def\Disc{{\rm Disc}}
\def\diag{{\rm diag}}

\def\Aut{{\rm Aut}}

\def\gen{{\rm gen}}

\def\sqf{{\rm sqf}}
\def\fund{{\rm fund}}

\def\Vol{{\rm Vol}}
\def\R{{\mathbb R}}
\def\A{{\mathbb A}}
\def\F{{\mathbb F}}
\def\FF{{\mathcal F}}

\def\Q{{\mathbb Q}}

\def\C{{\mathcal C}}

\def\Z{{\mathbb Z}}
\def\P{{\mathbb P}}
\def\F{{\mathbb F}}
\def\Q{{\mathbb Q}}
\def\C{{\mathbb C}}

\newtheorem{theorem}{Theorem}[section]
\newtheorem{corollary}[theorem]{Corollary}
\newtheorem{lemma}[theorem]{Lemma}

\newtheorem{remark}[theorem]{Remark}
\newenvironment{proof}{\noindent {\bf Proof:}}{$\Box$ \vspace{2 ex}}

\title{The geometric sieve and the density of squarefree values of invariant polynomials}

\author{Manjul Bhargava}

\begin{document}
\maketitle
\begin{abstract}
We 
develop a method 
for determining the density of squarefree values taken by certain
multivariate integer polynomials that are invariants for the action of
an algebraic group on a vector space.  The method is shown to apply to
the discriminant polynomials of various prehomogeneous and coregular
representations where generic stabilizers are finite.  This has
applications to a number of arithmetic distribution questions, e.g.,
to the density of small degree number fields
having squarefree discriminant, and the density of certain
unramified {\it nonabelian} extensions of quadratic fields. In
separate works, the method forms an important ingredient in
establishing lower bounds on the average orders of Selmer groups of
elliptic curves.
\end{abstract}

\setcounter{tocdepth}{2}

\section{Introduction}

The purpose of this article is to develop a method for determining the density of
squarefree values taken by certain multivariate integer polynomials that are invariants for 
an algebraic group acting on a vector space. 
 In the case of general polynomials in one or two variables having degree at most three or
six, respectively, methods of Hooley~\cite{Hooley} or
Greaves~\cite{Greaves}, respectively, may be applied; in other cases, if the degree of the polynomial is quite small relative to the number
of variables, then the circle method may be used to extract squarefree
values of the polynomial in question.  
In contrast, our method may be 
applied to polynomials of high degree---even when the degree and the
number of variables are comparable---so long as the polynomial has some
extra {structure}, such as symmetry under the action of a ``suitably large'' algebraic
group defined over $\Z$ (this condition will be made more precise in Section~2).

\subsection{The density of number fields having squarefree discriminant}

The most classical specific cases of arithmetic interest that our
method addresses is that of determining the density of small degree
number fields having squarefree discriminant.  Building on the works of
Levi~\cite{Levi}, Wright--Yukie~\cite{WY}, and
Gan--Gross--Savin~\cite{GGS}, it was shown in~\cite{DF}, \cite{hcl3},
and~\cite{hcl4} that
the integers that occur as the discriminants of orders in cubic,
quartic, and quintic number fields, respectively, correspond to
suitable integer values taken by certain fixed multivariate integral
polynomials $f_3$, $f_4$, and $f_5$, having degrees 4, 12, and 40 in
4, 12, and 40 variables, respectively.  This correspondence between
number field discriminants and integers represented by these special
polynomials was indeed what was used in~\cite{DH}, \cite{dodqf}, 
and~\cite{dodpf}, in conjunction with geometry-of-numbers arguments, to
determine the density of discriminants of cubic, quartic, and quintic
number fields, respectively.
 
To determine the density of such number fields having squarefree
discriminant, we must thus determine the density of squarefree integer
values taken by these special polynomials $f_3$, $f_4$, and $f_5$.
As we have noted, for general polynomials of large degree $d$ in about
$d$ variables, this is an unsolved problem.  However, using the
structure of these special polynomials---namely, that they are
invariants for the action of a ``suitably large'' algebraic group---we
determine in \S4 the density of squarefree values taken by these polynomials.

As a consequence, we prove that a positive density of all $S_n$-number fields of degrees $n=3$, $4$, and $5$ have squarefree discriminant, and we determine this density precisely.  We similarly determine the density of such number fields that have fundamental discriminant.
Specifically, we~prove:

\begin{theorem}\label{sqfdisc}
  Let $n=3$, $4$, or $5$, and let $N_n^\sqf(X)$ $($resp.\
  $N_n^\fund(X))$ denote the number of isomorphism classes of number
  fields of degree $n$ having squarefree 
$($resp.\ fundamental$)$ discriminant of absolute value less than $X$.  Then

\vspace{-.2in}
\begin{eqnarray*}
{\rm (a)} &
N_n^\sqf(X)
\,\,=\, 
\displaystyle\frac{r_2(S_n)}{3n!}\zeta(2)^{-1}\cdot X+o(X);
\\[.1in]
{\rm (b)} &
N_n^\fund(X)
=\, 
\displaystyle\frac{r_2(S_n)}{2n!}\zeta(2)^{-1}\cdot X+o(X),
\\[-.075in]
\end{eqnarray*}
where $r_2(S_n)$ denotes the number of $2$-torsion elements in the symmetric group $S_n$.
\end{theorem}
Note that Theorem~\ref{sqfdisc} is true also for $n=2$,
provided that we count each quadratic field~$K$ with weight $\frac{1}{2}$
(i.e., with weight $\frac{1}{\#\Aut(K)}$). We conjecture that Theorem~\ref{sqfdisc} holds for general~$n$.

In conjunction with the main results of \cite{DH}, \cite{dodqf}, and
\cite{dodpf}, which give the total density of discriminants of cubic,
quartic, and quintic fields, respectively, we conclude:

\begin{corollary}\label{sqfreedisccor}
When ordered by absolute discriminant, the proportion of $S_n$-number
fields of degree~$n$ $(n\in\{2,\ldots,5\})$ having 
fundamental discriminant is given by
$$\left\{\begin{array}{cl}
1 & \mbox{if $n=2$}\,; \\[.07in]
{\zeta(2)^{-1}}{\zeta(3)} & \mbox{if $n=3$}\, ;\\[.07in]
\zeta(2)^{-1}\prod_p(1+p^{-2}-p^{-3}-p^{-4})^{-1} & \mbox{if $n=4$}\, ;\\[.07in]
\zeta(2)^{-1}\prod_p(1+p^{-2}-p^{-4}-p^{-5})^{-1} & \mbox{if $n=5$}\,.
\end{array}\right.
$$ Furthermore, the proportion of $S_n$-number fields of degree~$n$
$(n\in\{2,\ldots,5\})$ having squarefree~discriminant is~exactly~$2/3$
of the proportion having fundamental discriminant. 
\end{corollary}

\vspace{.1in}
Both Theorem~\ref{sqfdisc} and Corollary~\ref{sqfreedisccor} follow from a general theorem that our methods allow us to prove, concerning the asymptotic count of $S_n$-number fields of degree $n\leq 5$ satisfying any desired finite or suitable infinite set of local conditions:  

\begin{theorem}\label{gensqfree}
Let $n=2$, $3$, $4$, or $5$.  Let $\Sigma=(\Sigma_\infty,\Sigma_2,\Sigma_3,\ldots)$ 
denote an acceptable set of local specifications for degree $n$ extensions of $\Q$, i.e., $\Sigma_\nu$ is any subset
of $($isomorphism classes of$)$ \'etale degree $n$ extensions of $\Q_\nu$ for each place $\nu$ of $\Q$, such that for sufficiently large primes $p$, the set $\Sigma_p$ contains all \'etale extensions $K_p$ of $\Q_p$ of degree $n$ such that $p^2\nmid \Disc(K_p/\Q_p)$.  Let $N_{n,\Sigma}(X)$ denote 
the number of $S_n$-number fields $K$ of degree $n$ having
absolute discriminant at most $X$ such that
$K\otimes\Q_\nu\in\Sigma_\nu$ for all places $\nu$ of $\Q$.
Then
\begin{equation}\label{sigmalimit}
\lim_{X\rightarrow\infty} \frac{N_{n,\Sigma}(X)}{X} = 
\displaystyle{\Bigl(
\sum_{K\in\Sigma_\infty}\frac12\cdot
\frac1{\#\Aut(K)}\Bigr)\prod_p
\Bigl(\sum_{K\in\Sigma_p}}
\frac{p-1}p\cdot\frac1{\Disc_p(K)}\cdot\frac1{\#\Aut(K)} \Bigr).
\end{equation}
\end{theorem}
The above theorem thus allows one to count number fields of degree at
most five satisfying very general sets of local conditions.  In
particular, it proves a more general version (namely, where we allow
infinitely many local conditions) of the heuristics given
in \cite[(4.2)]{imrn}.  

Since having squarefree or fundamental discriminant is a local
condition of the type occurring in Theorem~\ref{gensqfree},
Theorem~\ref{sqfdisc} will follow from Theorem~\ref{gensqfree} once
the sums in the Euler factors in~(\ref{sigmalimit}), i.e., the {\it
  local masses}, are computed (see~\S4 for details).

\subsection{Unramified nonabelian ($A_n$- and $S_n\times C_2$-) extensions of quadratic fields}

The density of degree $n$ number fields having squarefree discriminant is directly related to the
distribution of certain unramified {\it nonabelian} extensions of
quadratic fields.  
More precisely, given a finite group $G$ and a
quadratic field $K$, we may consider the set $U(K;G)$ of all
isomorphism classes of unramified {\it $G$-extensions of $K$}, i.e., Galois extensions of $K$ with
Galois group $G$.  An extension
$L\in U(K;G)$ is not necessarily normal over $\Q$, and its normal
closure over $\Q$ has Galois group $G'\subset G\wr C_2=(G\times
G)\rtimes C_2$.  It is thus natural to partition $U(K;G)$ into the
sets $U(K; G,G')$, where $U(K; G,G')$ denotes the set of all
isomorphism classes of unramified $G$-extensions $L$ of $K$ such that
the Galois closure of $L$ over $\Q$ has Galois group $G'$.  If $L\in
U(K; G,G')$, then we say that $L$ is an unramified extension of $K$ of
{\it type $(G,G')$}, or simply an unramified {\it $(G,G')$-extension}.

\begin{theorem}\label{unrexts}
  Let $n=3,$ $4$, or $5$, and let $E^+(G,G')$ $($resp.\ $E^-(G,G'))$
  denote the average number of unramified $(G,G')$-extensions that
  real $($resp.\ imaginary$)$ quadratic fields possess, where
  quadratic fields are ordered by their absolute discriminants.  Then
\begin{equation*}
\begin{array}{clcc}
{{\rm (a)}} & E^+(A_n,S_n) &\!=\!& \displaystyle{\frac{1}{n!}\,;}\\[.185in]
{{\rm (b)}} & E^-(A_n,S_n) &\!=\!& \displaystyle{\frac{1}{2(n-2)!}\,;}\\[.25in]
{{\rm (c)}} & E^+(S_n,S_n\times C_2) &\!\,=\,\,\!& \infty\,;\\[.245in]
{{\rm (d)}} & E^-(S_n,S_n\times C_2) &\!\,=\,\,\!& \infty\,.\\[0in]
\end{array}
\end{equation*}
\end{theorem}
In other words, the average number of unramified $A_n$-extensions
($n=3$, 4, or 5) possessed by real or imaginary quadratic fields is
positive, and the average number of unramified $S_n\times
C_2$-extensions is also positive, and in fact infinite!  For $n=2$, note that Theorem 1 is still true, except that the constants
in (a) and (b) must each be multiplied by 2, again reflecting the fact that
a quadratic extension has two automorphisms.

The case $n=3$ in Theorems~\ref{unrexts}(a)--(b) corresponds to
abelian ($A_3$-) extensions, and is due to
Davenport--Heilbronn~\cite{DH}, who obtained these results via the
use, in particular, of methods that amount essentially to class field
theory (see \cite{DW} for this nice interpretation).  The cases $n=4$
and $n=5$ of Theorems~\ref{unrexts}(a)--(b) are both new, and to our knowledge are independent of and cannot be treated by class field theory.  Indeed, they yield information on the
distribution of certain nonabelian unramified extensions of quadratic
fields, namely, those corresponding to the groups $A_4$ and $A_5$;
in particular, the case $n=5$ yields information about the
distribution of unramified extensions of a quadratic field of a
nonsolvable type, namely $A_5$. Theorems~\ref{unrexts}(c)--(d) are also new.

 Returning to the statement of Theorem~\ref{unrexts}, it is an interesting
 question as to which groups $G,G'$ lead to quantities 
 $E^+(G,G')$ and $E^-(G,G')$ that exist and are finite, and what their
 values are when they are finite.  Both possibilities of finite and
 infinite already occur in Theorem~\ref{unrexts}.  In the case of {\it
   abelian} $G$, we must have that $G'=G\rtimes C_2$ (where the
 nontrivial element of $C_2$ acts on $G$ by inversion).  The
 Cohen--Lenstra heuristics~\cite{CL} can then be shown to imply that, for $G$ abelian, 
\begin{eqnarray}
\label{CLeq1}
E^+(G,G\rtimes C_2) &\!=\!& \displaystyle{\frac{1}{\,|\Aut(G)|\cdot|G|\,}\,,}\\[.05in]
\label{CLeq2}
E^-(G,G\rtimes C_2) &\!=\!&  \;\;\;\;\;\displaystyle{\frac{1}{|\Aut(G)|}\,}
\end{eqnarray}
whenever $|G|$ is odd.

Note that the cases in Theorems~\ref{unrexts}(a)--(b) in which $G$ is abelian occur
when $n=3$, and in these cases the values agree with those predicted
by (\ref{CLeq1}) and (\ref{CLeq2}). 
It would be interesting to have more general
heuristics for $E^\pm(G,G')$ that include both the abelian results and
conjectures above
as well as the nonabelian results of Theorem~\ref{unrexts}.

In parts (c) and (d) of Theorem~\ref{unrexts}, it is actually possible to say
something more precise; namely, the methods of Section~4 show that
\begin{eqnarray}
\label{mp1}
\displaystyle{\sum_{0<\Disc(K)<X}|U(K;S_n,S_n\times C_2)|} &\sim&
c_n^+ X\log\,X\,;\\[.1in]
\label{mp2}
\!\!\displaystyle{\sum_{-X<\Disc(K)<0}\!\!\!\;|U(K;S_n,S_n\times C_2)|}  &\sim& \;\!c_n^- X\log\,X\,,
\end{eqnarray}
for $n=3$, $4$, and $5$, where $c_n^\pm$ are certain positive
constants which depend on $n$.  

\subsection{Squarefree values taken by polynomials such as $f_3$, $f_4$, and $f_5$}

As we have mentioned, 
to prove Theorem~\ref{sqfdisc}, Corollary~\ref{sqfreedisccor}, Theorem~\ref{gensqfree}, and Theorem~\ref{unrexts}, one must
determine the densities of lattice points in $\R^m$ where the values of certain
polynomials---namely, the {discriminant} polynomials $f_3$, $f_4$, or $f_5$---are squarefree.
In general, counting the number of lattice points of bounded height
where a polynomial takes squarefree values is an unsolved problem,
although conjecturally it is easy to guess what should happen.
Namely, if $f(x_1,\ldots,x_m)$ is any squarefree polynomial over $\Z$
then, barring congruence 
obstructions, one expects that $f$ takes infinitely many squarefree values on 
$\Z^m$.  More precisely, one expects
\begin{equation}\label{localdensity}
\lim_{N\to\infty} \frac{\#\{x\in \Z^m\cap[-N,N]^m:f(x) \mbox{
    squarefree}\}}{(2N+1)^m} = \prod_p(1-c_p/p^{2m}),
\end{equation}
where, for each prime $p$, the quantity $c_p$ is the
number of elements $x\in(\Z/p^2\Z)^m$ satisfying $f(x)=0$ in $\Z/p^2\Z$.

When $m=1$, this assertion is relatively easy to prove in degrees $\leq 2$, while 
for cubic polynomials it was proven by Hooley~\cite{Hooley}.  For degrees $\geq 
4$, it appears that no single example is known of a univariate irreducible 
polynomial $f$ 
satisfying (\ref{localdensity})!  As for polynomials in more than one variable, 
Greaves has shown that (\ref{localdensity}) holds for all binary forms of
degree at most 6.  

Conditionally, Granville~\cite{Granville} showed that 
(\ref{localdensity}) follows, for all univariate polynomials of any degree, 
from the ABC Conjecture.  More recently, Poonen~\cite{Poonen2} proved that the ABC 
Conjecture implies that (a slightly weaker version of) equation (\ref{localdensity}) 
is true also for all multivariate polynomials. 

In this article, we give three special examples of polynomials $f$ for which
we can prove unconditionally that (\ref{localdensity}) holds; namely,
these are the three polynomials that we use to prove Theorem~\ref{sqfdisc}, Corollary~\ref{sqfreedisccor}, and Theorems~\ref{gensqfree}--\ref{unrexts}. 
More precisely, let $f$ $(=f_3$, $f_4$, or $f_5$) denote the primitive integral polynomial that
generates the ring of invariants for:
\begin{itemize}
\item[(i)]
the action of $\SL_2(\C)$ on $\Sym_3(\C^2)$, the space of binary cubic forms over $\C$;
\item[(ii)]
the action of $\SL_2\times\SL_3(\C)$ on $\C^2\otimes \Sym_2(\C^3)$, the space of pairs of ternary quadratic forms over $\C$; or
\item[(iii)]
the action of $\SL_4\times\SL_5(\C)$ on $\C^4\otimes \wedge^2(\C^5)$, the space of quadruples of $5\times 5$ skew-symmetric matrices over $\C$,
\end{itemize}
respectively.  
Then for (i), (ii), or (iii), $f$ is a polynomial of degree~$m$ in $m$ variables, where $m=4$, 12, or 40,  respectively (see \cite{SK}, or see \cite{hcl3}--\cite{hcl4} for explicit constructions of these invariant polynomials).  We prove:

\begin{theorem}\label{fthm}
The polynomials $f$ in {\rm (i)--(iii)} above are each irreducible over $\bar\Q$ and are of degree~$m$ in $m$ variables, where $m=4$, $12$, and $40$ respectively.  Moreover, for each of these polynomials $f$, we have 
\begin{equation*}\label{localdensity2}
\lim_{N\to\infty} \frac{\#\{x\in \Z^m\cap[-N,N]^m:f(x) \mbox{ \rm
    squarefree}\}}{(2N+1)^m} = \prod_p(1-c_p/p^{2m}) 
    = \frac23\zeta(2)^{-1},
\end{equation*}
where $c_p
$ is the number of elements $x\in(\Z/p^2\Z)^m$ satisfying $f(x)=0$ in $\Z/p^2\Z$.
 \end{theorem}

 For these three discriminant polynomials $f$, particularly in the
 cases (ii) and (iii) where the degrees are large ($\geq 4$) in each
 individual variable (and the number of variables is equal 
 to the degree), we do not believe that any of the previously known
 unconditional results and methods as described above would apply.
 Thus these $f$ give new examples of polynomials satisfying (\ref{localdensity}).  It is interesting to note that
 the density of squarefree values taken by each of these three
 discriminant polynomials $f$ is exactly $\frac23\zeta(2)^{-1}$, 
independent of $f$.
 
 \subsection{Squarefree values of discriminants of genus one models}
 
 The method, which we will describe more axiomatically in the next subsection and in Section 2, 
 may also be applied to various other polynomials that are invariant under the action of a suitably large
 algebraic group defined over $\Z$.  Another family of classical examples on which the method applies are the discriminant polynomials of models of genus one curves. 

There are many such models of genus one curves of interest.  Genus one curves with maps to $\P^1$, $\P^2$, $\P^3$, or $\P^4$, via complete linear systems of degrees 2, 3, 4, or 5, are called {\it genus one normal curves} of degree 2, 3, 4, or 5, respectively.
They can be realized as: a double cover of~$\P^1$ ramified at four
points; a cubic curve in $\P^2$; the intersection of a pair of
quadrics in $\P^3$; or the intersection of five quadrics in $\P^4$
arising as the $4\times 4$ sub-Pfaffians of a $5\times 5$
skew-symmetric matrix of linear forms on $\P^4$.  A genus one model of
degree one may be viewed simply as an elliptic curve in Weierstrass
form.  (See, e.g., \cite{Fisher1} for a beautiful exposition.)

We note that we may also consider genus one models in products of
projective spaces.  
For example, 
a genus one curve in $\P^1\times\P^1$ is cut out by a bidegree $(2,2)$-form
on $\P^1\times\P^1$; and a genus one curve in
$\P^2\times\P^2$ is similarly cut out by three bidegree $(3,3)$-forms on
$\P^2\times \P^2$. 
These cases will be carried out in more detail in \cite{BH2}.
(See \cite{BH} 
also for other 
examples of 
such
spaces of genus one models.)

For all these genus one models over $\Z$, we show that the discriminant polynomials of these genus one curves all take the expected (positive) densities of squarefree values.   (Recall that the discriminant of a genus one model is the polynomial whose nonvanishing is equivalent to the smoothness of the corresponding genus one curve.)  For genus one models of degree one, i.e., Weierstrass elliptic curves $y^2=x^3+Ax+B$, the result is easy, as the discriminant polynomial $-4A^3-27B^2$ is only of degree 2 as a polynomial in $B$.  For higher degree genus one models, the result is much more difficult to obtain.  

More precisely, let $g$ (which we will denote by $g_2$, $g_3$, $g_4$,
$g_5$, 
respectively) denote the primitive integral discriminant polynomial of any of the following representations: 
\begin{itemize}
\item[(i)]
the action of $\SL_2(\C)$ on $\Sym^4(\C^2)$, the space of binary quartic forms over $\C$;
\item[(ii)]
the action of $\SL_3(\C)$ on $\Sym^3(\C^3)$, the space of ternary cubic forms over $\C$;
\item[(iii)]
the action of $\SL_2\times\SL_4(\C)$ on $\C^2\otimes \Sym^2(\C^4)$,
the space of pairs of quaternary quadratic forms over $\C$; 
\item[(iv)]
the action of $\SL_5\times\SL_5(\C)$ on $\C^4\otimes \wedge^2(\C^5)$, the space of quintuples of $5\times 5$ skew-symmetric matrices over $\C$,
\end{itemize}
respectively.  Then the discriminant polynomial $g$ on each of these
representations detects stable orbits, i.e., $g$ does not vanish
precisely when the orbit is closed and has finite stabilizer.  The
discriminant $g$ of an element in any of these representations also
corresponds to the discriminant of the associated genus one model,
i.e., $g$ does not vanish precisely when this associated genus one
curve is smooth.

The dimensions of the representations in (i)--(iv) above are given by
$5$, $10$, $20$, and $50$, 
respectively, while the
degrees of the corresponding discriminant polynomials are given by
$6$, 12, 24, and 60, 
respectively (see~\cite{Fisher1} 
for explicit constructions of these invariant
polynomials). Then we prove:

\begin{theorem}\label{gthm}
The polynomials $g$ in {\rm (i)--(iv)} above are each irreducible.  Moreover, for each of these polynomials $g$, we have 
\begin{equation*}\label{localdensity3}
\lim_{N\to\infty} \frac{\#\{x\in \Z^m\cap[-N,N]^m:g(x) \mbox{ \rm
    squarefree}\}}{(2N+1)^m} = \prod_p(1-c_p/p^{2m}) 
\end{equation*}
where $c_p$
denotes the number of elements $x\in(\Z/p^2\Z)^m$ satisfying $g(x)=0$ in $\Z/p^2\Z$.
 \end{theorem}
Thus a positive density of genus one models over $\Z$ mapping into
$\,\P^1$, $\,\P^2$, $\,\P^3$, or $\,\P^4$, 
have squarefree discriminant.  In particular, a positive density of
binary quartic forms over $\Z$, and a positive density of ternary cubic forms
over $\Z$, have squarefree discriminant.

These results, and the methods behind them, play an important role in
establishing lower bounds on the average sizes of Selmer groups of
families of elliptic curves in \cite{BS,TC,foursel,fivesel} and in \cite{BH2}.  They also play a key role in proving that
the local--global principle fails for a positive proportion of plane
cubic curves over $\Q$ (see \cite{hasse}).
 
 \subsection{Method of proof}
 
 Let $f$ be an integral polynomial on $V(\Z)\cong \Z^m$.  As is
 standard in squarefree sieves (see, e.g., \S3.4 for more details), the
 equality (\ref{localdensity}) can be proven for $f$ whenever
 sufficiently good upper bounds on sums involving $w_p(f,H)$ are
 obtained, where $w_p(f,H)$ denotes the number of points $v\in V(\Z)$
 having height at most $H$ (= the maximum of the absolute values of
 the coordinates) satisfying $p^2\mid f(v)$.  It is natural to
 partition the set $W_p=W_p(V) \subset V(\Z)$ of elements $v\in V(\Z)$
 such that $p^2\mid f(v)$ into two sets: $W_p^{(1)}$, consisting of
 elements $v\in V(\Z)$ on which $f$ vanishes modulo $p^2$ for ``mod $p$
 reasons'', i.e., $f(v')\equiv 0$ (mod $p^2$) for any $v'\equiv v$
 (mod $p$); and $W_p^{(2)}$, consisting of the elements $v\in V(\Z)$
 on which $f$ vanishes modulo $p^2$ for ``mod $p^2$ reasons'', i.e.,
 there exist $v'\equiv v$ (mod $p$) such that $f(v')\not\equiv 0$ (mod
 $p^2$).

 As a consequence, we may write $w_p(f,H)$ as a sum
 $w_p^{(1)}+w_p^{(2)}$, where $w_p^{(i)}$ denotes the portion of the
 count of elements in $W_p$ coming from $W_p^{(i)}$.  It is well-known
 that good estimates on the relevant sums involving $w_p^{(1)}$ can be
 obtained by ``geometric sieve'' or ``closed-point sieve'' methods
 (the latter terminology is due to Poonen), as introduced in the work
 of Ekedahl~\cite{Ek}; see also Poonen~\cite{Poonen,Poonen2} for a
 very clear  
 treatment.  We will prove a precise and quantitative version of
 Ekedahl's sieve estimates in \S3.2, which will be useful in the
 applications.

 The difficulty in squarefree sieves for values taken by integral
 polynomials thus arises in the estimation of sums involving
 $w_p^{(2)}$.  It is essentially here that Granville~\cite{Granville}
 and Poonen~\cite{Poonen2} use the ABC Conjecture to obtain the desired
 estimates.  For the polynomials arising in
 Theorems~\ref{fthm} and~\ref{gthm}, we sidestep the use of the ABC
 Conjecture by using instead the invariance of these polynomials under
 the action of an algebraic group $G$ defined over $\Z$. Specifically,
 for the polynomials $f$ arising in Theorem~\ref{fthm}, we show that
 for any element $v\in W_p^{(2)}$, there always exists an element
 $\gamma\in G(\Q)$ such that $\gamma v\in V(\Z)$ and $f(\gamma
 v)=f(v)/p^2$. Together with estimates from the geometry-of-numbers in
 \cite{Davenport2,dodqf,dodpf} giving uniform upper bounds on the
 number of ``irreducible'' $G(\Z)$-classes on $V(\Z)$ having bounded
 absolute discriminant, this is sufficient to obtain the desired upper
 bounds on~$w_p^{(2)}$.

 With a related construction, for
all but one of the polynomials $g$ arising
in Theorem~\ref{gthm} we show that for any element $v\in W_p^{(2)}$,
there always exists an element $\gamma\in G(\Q)$ such that $\gamma
v\in W_p^{(1)}$ and $f(\gamma v)=f(v)$; i.e., via the action of $G(\Q)$,
we turn $v\in V(\Z)$ on which $f$ vanishes modulo~$p^2$ for mod~$p^2$ reasons into $v'$ for which $f$ vanishes modulo~$p^2$ for mod
$p$ reasons!  As before, we combine this construction with estimates
from the geometry-of-numbers as in \cite{BS,TC,foursel,fivesel},
which give uniform upper bounds on the number of ``irreducible''
$G(\Z)$-classes on $V(\Z)$ having bounded absolute discriminant, to
deduce the desired upper bounds on~$w_p^{(2)}$.

In 
Case (i) of Theorem~\ref{gthm},
however,
this argument does {not} work; we find that the group~$G(\Q)$ in 
this case is just too small to do the job.  We get around this problem
via a further argument that we call the ``embedding sieve''. Namely,
we find a representation~$G'$ on~$V'$, defined over $\Z$, and an
invariant polynomial $f'$ for this action, such that: there is a map
of orbits $\phi:G(\Z)\backslash V(\Z)\to G'(\Z)\backslash V'(\Z)$,
having preimages of absolutely bounded cardinality, for which $f'(\phi(v))=f(v)$.
Furthermore, we choose $(G',V')$ such that $G'(\Q)$ is sufficiently
larger than $G(\Q)$, while the set of irreducible orbits of
$G'(\Z)$ on $V'(\Z)$ is not too large; this allows one to obtain an
estimate $w_p^{(2)'}$ on $V(\Z)'$, which then leads to a good estimate
also for $w_p^{(2)}$. Amusingly, in the case of $g_2$ in Theorem~\ref{gthm}, we embed
$(G,V)$ into the representation $(G',V')$ corresponding to the
polynomial $f_4$ in Theorem~\ref{fthm}!

Indeed, the latter argument (which will be described in more detail in
\S5) shows that the method of this paper may in fact be applied to 
some polynomials $f$ that do not have a very large group of symmetries; in
such cases, we simply attempt to arrange a suitable embedding where
the method does apply to give the desired estimates.  Although we only
apply this embedding sieve in 
one case in this paper, it will serve
as a starting point in a sequel to this paper where we study
squarefree values of more general polynomials that may have fewer
symmetries.

This paper is organized as follows.  In Section 2, we enumerate 
a natural set
of axioms on an integral multivariate polynomial $f$
which is sufficient to deduce that $f$ takes the expected
density of squarefree values (i.e., $f$ satisfies
(\ref{localdensity})).  In Section 3, we then prove 
the latter assertion, 
by developing the geometric sieve method that we use to
extract squarefree values of such polynomials
satisfying these axioms.
Finally, in Sections 4 and 5, we then prove Theorems~\ref{fthm}
and~\ref{gthm}, by proving that all but one of the polynomials
occurring in these theorems satisfy the axioms of Section 2.  For the
remaining polynomial $g_2$, we describe an extension of these axioms
(the ``embedding sieve'') that allows us to prove (\ref{localdensity})
also for this polynomial.

\section{Some general criteria for extracting squarefree values of invariant polynomials}

Let $V$ be a representation of an algebraic group $G$ defined over $\Z$, and let $f$ be an integer polynomial of degree $d$ that is a relative invariant for the action of $G$ on $\Z$ and whose squarefree values we wish to extract.  Let $m:=\dim(V)$.  We use $G^1$ to denote the kernel of the determinant map $G\to \GL(V)\to \mathbb G_{\rm m}$. 

Suppose $f$, $G$, and $V$ have the following properties:
\begin{enumerate}
\item 
There is a 
notion of a {\it generic} element of $V(\Z)$; the subset $V(\Z)^\gen$ 
of generic elements in $V(\Z)$ is $G(\Z)$-invariant, and satisfies 
$$\mu(V(\Z)^\gen):=\lim_{N\to\infty} \frac{\#\{x\in V(\Z)^\gen\cap[-N,N]^m\}}{(2N+1)^m} = 1.$$
\item 
The order of the stabilizer  in $G(\bar\Q)$ of any element in $V(\Z)^\gen$ is finite and absolutely bounded.
\item There is a continuous (but not necessarily polynomial) invariant $I$ for the action of $G^1(\Z)$ on $V(\Z)$ that is homogeneous of degree $d$, i.e., $I(\lambda v)=\lambda^dI(v)$.
\item There is a fundamental domain $\FF$ for the action of $G^1(\Z)$ on $V(\R)$ such that 
the region $\FF_X:=\{v\in \FF:|I(v)|<X\}$ is measurable and homogeneously expanding, i.e., $\FF_X=X^{1/d}\FF_1$, and 
the volume $\Vol(\FF_X)$
of $\FF_X$
is finite. 
\item For any subset $S$ of $V(\Z)$ defined by congruence conditions modulo finitely many prime powers, we have 
\begin{equation}\label{nsxeq}
N(S;X):=\#\{v\in S\cap \FF_X \mbox{ generic}\} = \Vol(\FF_X)\cdot \prod_p\mu_p(S)+o(X^{m/d}),
\end{equation}
where $\mu_p(S)$ denotes the density of the $p$-adic closure of $S$ in $V(\Z_p)$.
\item Fix a prime $p$. If $v\in V(\Z)^\gen$ is an element such that $f(v)$ is a multiple of
  $p^2$, then there is a nonnegative real number $a=a_{v}$, an
  absolutely bounded integer $k=k_{v}\geq 0$, and an element $g=g_{v}\in
  G(\Q)$ such that
\begin{itemize}\item[{\rm (i)}] $|I(gv)|=p^{-a}|I(v)|$; 
 \item[{\rm (ii)}] the element $gv$ lies in $V(\Z)^\gen$ in the reduction (mod $p$) of a closed $G$-invariant subscheme $Y_k$ of $V$ (viewed as affine $n$-space) defined over $\Z$, depending only on $k$, that has codimension~$\geq k$; 
 \item[{\rm (iii)}] for each fixed $k$, every point of $Y_k(\Z)$ arises as $g_{v}v$ for some $v\in V(\Z)^\gen$ at most $c$ times up to $G(\Z)$-equivalence, where $c$ is an absolute constant;
 \item[{\rm (iv)}]  $\displaystyle\frac m d\cdot a+k-1$ is
   bounded below by an absolute positive constant $\eta$.
 \end{itemize}
\end{enumerate}

\begin{theorem}\label{maingeneral}
  If $f$, $G$, $V$ satisfy Conditions $1$--$6$, then $f$ takes the
  expected density of squarefree values, i.e., $f$ satisfies
  $(\ref{localdensity})$.
\end{theorem}

While Conditions 1--6 may seem very restrictive, we will see in
Sections 4 and 5 that they are satisfied by all but 
one of the
polynomials in Theorems~\ref{fthm} and \ref{gthm} (and, indeed, by
many other polynomials, e.g., by a number of the discriminant polynomials 
occurring in \cite{BH}).
In general, the notion of {\it generic} in Condition 1 is chosen so
that the cusps of the fundamental domain $\FF$ in Condition 4 contain
mostly non-generic points.  Indeed, the integral points in the cusps
of such fundamental domains $\FF$ tend to lie primarily on certain
subvarieties; the lattice points in $V(\Z)$ that lie outside the union
of these subvarieties are then called {\it generic}.  Condition 2 is
of course common, and will generally be satisfied in any
representation that has stable orbits in the language of geometric
invariant theory.  With Conditions 1--4 satisfied, Condition 5 can
then be proven using geometry-of-numbers methods (as developed, e.g.,
in the works \cite{Davenport2,dodqf,dodpf,BS}).

Finally, Condition 6 is also true for 
all the representations and
polynomials in Theorem~\ref{fthm} and 
all but one of the representations and polynomials in
Theorem~\ref{gthm}.  
However, it is not true for the very first polynomial $g_2$ in Theorem~\ref{gthm}. 
In general, Condition~6 can be quite
restrictive (as opposed to Conditions 1--5), because there may not be
enough symmetries in~$G$ to satisfy the condition.  In such cases, we
may attempt to embed~$V$ into a larger representation that has more
symmetries for which (a suitable version of) Condition~6 {\it is}
satisfied!  This argument indeed works for the remaining 
representation, 
and will be important in future applications.

\begin{remark}\label{fxversion}{\em 
In the course of proving Theorem~\ref{maingeneral},
we will also show that the polynomials $f$ in this theorem---in
addition to satisfying (\ref{localdensity})---also satisfy
\begin{equation}\label{ld2}
\lim_{X\to\infty} \frac{\#\{x\in \FF_X\cap V(\Z)^\gen:f(x) \mbox{ squarefree }\}}{\#\{x\in \FF_X\cap V(\Z)^\gen\}}
=\prod_p(1-c_p/p^{2m})
\end{equation}
where again $c_p$ denotes the number of elements $x\in(\Z/p^2\Z)^m$ satisfying $f(x)=0$ in $\Z/p^2\Z$.}
\end{remark}

\section{A geometric squarefree sieve}

In this section, we describe the geometric sieve method that we use to
extract squarefree values of polynomials.  

In \S3.1--3.2 (see in particular Theorem~\ref{gsthm}), we lead up to a
quantitative version of a certain uniformity estimate due to
Ekedahl~\cite{Ek} (see also Poonen~\cite{Poonen,Poonen2} and
Poonen--Stoll~\cite{PS}).  Ekedahl shows that, in appropriate
situations, the usual inclusion--exclusion tail becomes ``negligible''
as the cut off defining the tail gets larger and larger.  For our
applications here as well as in future applications, we require
precise quantitative versions of these tail estimates (i.e., how
negligible is ``negligible''?), when counting lattice points in
homogeneously expanding regions.  Given any variety of codimension at
least 2 defined over $\Z$, these estimates will, in particular, yield
a method for sieving out those lattice points that, for some
sufficiently large~$p$, reduce (mod~$p$) to a point on the reduction
of that variety (mod~$p$).  The quantitative versions of the relevant
tail estimates that we prove in \S3.2 enable one also to obtain second
order terms or power-saving error terms in the applications.

In this article, we are particularly concerned with sifting out those
lattice points on which a given polynomial takes non-squarefree
values.  The application to this scenario is described in
Subsection~\ref{sqfreeapp}. In the final Subsection~3.4, we then prove
Theorem~\ref{maingeneral}
of Section~2, namely, that any integral polynomial $f$ satisfying the axioms of
Section~2 takes the expected number of squarefree values. 

\subsection{The number of lattice points in a homogeneously 
expanding region lying on a subvariety}

We start with the following simple and oft-used lemma that states that
the number of lattice points on a given variety in a homogeneously
expanding region in $\R^n$ grows at most polynomially in the linear
scaling factor, where the degree of the polynomial is the dimension of
the variety.  Though this result is well-known, we include a proof
here for completeness, and as a preparatory ingredient for the sieve
estimates in \S3.2.

\begin{lemma}\label{first}
  Let $B$ be a compact region in $\R^n$ having finite measure.  Let
  $Y$ be a variety in $\R^n$ of codimension $k\geq1$.
Then 
we have
\begin{equation}
\#\{a\in rB\cap Y\cap \Z^n\} \,=\, O(r^{n-k}),
\end{equation}
where the implied constant depends only on $B$ and on $Y$.
\end{lemma}

\begin{proof}
We may clearly assume that $Y$ is irreducible, for otherwise we could simply sum over
the irreducible components of $Y$.  Since $Y$ has codimension $k$ in $\R^n$, there exist
polynomials $f_1,\ldots,f_k$ for which
\begin{equation}
Y \subseteq 
Y'':=
\{a\in \R^n \,|\, f_1(a)=f_2(a) = \cdots = f_k(a) = 0\}
\end{equation}
such that the irreducible component $Y'$ of $Y''$ containing $Y$ 
also has codimension $k$.

We prove the estimate of Lemma~\ref{first} for $Y'$ in place of $Y$,
by induction on $n$, using the polynomials $f_1,\ldots,f_k$. 
 We always write the $f_i$ as polynomials in the arguments
$x_1,\ldots,x_n$.   
For the proof, we may clearly assume that each
$f_i$ is irreducible, for otherwise we can simply replace each $f_i$ by the irreducible factor of $f_i$ 
that vanishes on $Y'$. 
If all the $f_i$ do not involve some variable, say $x_n$, then the result follows by the induction hypothesis. 
So we may assume that every variable $x_1,\ldots,x_n$ occurs in at least one $f_i$. 

We now show, via elimination theory, that we may reduce to the case
where $k-1$ of the~$f_i$, say $f_1,\ldots,f_{k-1}$, all do not involve
some fixed variable, say $x_n$.  Indeed, by reordering the $f_i$ if necessary, let us assume that $f_k$ is nonconstant as a
polynomial in $x_n$.  Let $R_n(f_i,f_k)$ denote the resultant of $f_i$ and $f_k$ with respect to $x_n$. 
Since $R_n(f_i,f_k)=A_if_i+B_if_k$ for some polynomials $A_i$ and $B_i$ with $A_i$ nonzero, 
the irreducible component containing $Y$ in the variety cut out by 
$R_n(f_1,f_k),\ldots,R_n(f_1,f_{k-1}),f_k$ (= the variety cut out by $A_1f_1,\ldots,A_{k-1}f_{k-1},f_k$) is still $Y'$.
Thus we may simply replace each $f_i$ involving $x_n$ (for $i\in\{1,\ldots,k-1\}$) by $R_n(f_i,f_k)$, and we see that the irreducible component containing $Y$ of the new $Y''$ cut out by the new $f_i$ is still the variety $Y'$ 
of codimension $k$, where now $f_1,\ldots,f_{k-1}$ do not involve $x_n$.

Thus it suffices to prove the lemma when $f_1,\ldots,f_{k-1}$ are
polynomials only in $x_1,\ldots,x_{n-1}$.
Let $h_k$ denote the leading coefficient of $f_k$ as a polynomial in
$x_n$, so $h_k$ is a polynomial in $x_1,\ldots,x_{n-1}$. We may assume that $h_k$ does not 
vanish on $Y'$,
for otherwise we might as well eliminate the leading term of $f_k$, and
$f_1,\ldots,f_k$ would still cut out a variety $Y''$ whose irreducible component containing $Y$ is $Y'$.
Let $Z$ be the union of the irreducible components intersecting $Y'\cap\{h_k=0\}$ of the variety cut out by $f_1,\ldots,f_{k-1},h_k$ in~$\R^{n-1}$.  Then $Z$ is of codimension~$k$ in~$\R^{n-1}$.

We now partition $\Z^n$ into two sets of points: those on which $h_k$ vanishes and those on which it does not.  For the set of points where $h_k$ vanishes, we have
\begin{equation}
\#\{a\in rB\cap Y\cap \Z^n\,|\, h_k(a)=0\} \,=\, O(r^{n-1-k})\cdot O(r) \,=\, O(r^{n-k}),
\end{equation}
since there are at most $O(r^{n-1-k})$ eligible values for the first
$n-1$ arguments by the induction hypothesis applied to $Z$,
and then there are at most $O(r)$ possible values for the last coordinate of a point $a\in rB$.  

To handle  the points where $h_k$ does not vanish, if $d$ denotes the degree of $f_k$ as a polynomial in
$x_n$, then once values of $x_1,\ldots,x_{n-1}$ are fixed satisfying
$h_k(x_1,\ldots,x_{n-1})\neq 0$, then there are at most $d$ values for
$x_n$ satisfying $f_k(x_1,\ldots,x_n)=0$.  Therefore, using again the
induction hypothesis on the irreducible component containing $Y'$ of $\{f_1=\cdots=f_{k-1}=0\}$ of codimension $k-1$, we see that
\begin{equation}
\#\{a\in rB\cap Y\cap \Z^n\,|\, h_k(a)\neq 0\} \,=\, O(r^{(n-1)-(k-1)})\cdot d \,=\, O(r^{n-k}),
\end{equation}
as desired.
\end{proof}

\begin{remark}{\em Note that the $O(r^{n-k})$ estimate of Lemma~\ref{first} is
    optimal and is achieved for varieties of degree 1.  In the case of
    varieties of degree $>1$, the result can be improved, in some
    cases significantly, depending on the variety; see
e.g., the work of Heath-Brown~\cite{HB} and more recently of 
Salberger and Wooley~\cite{Sal}. We do not
    include their results here because, for our particular application,
    we require a result that includes also the case of degree 1.}
\end{remark}

\subsection{The number of lattice points in a homogeneously expanding
  region that reduce modulo some sufficiently large prime $p$ to a
  point on the reduction of a given variety modulo $p$}

We are now in a position to prove an upper asymptotic estimate for the
number of lattice points in a homogeneously expanding region in $\R^n$
that reduce modulo $p$, for some sufficiently large $p>M$, to an
$\F_p$-point of a given variety $Y\subset \A^n$ defined over $\Z$.  In
\cite{Ek}, Ekedahl proved that the
asymptotic proportion of such points in a box of sidelength $r$, as
$r\to\infty$, approaches 0 as $M\to\infty$, assuming that 
$Y$ has codimension at least two.
 
 We prove here the following precise quantitative version of Ekedahl's result:

\begin{theorem}\label{gsthm}
  Let $B$ be a compact region in $\R^n$ having finite measure, and let
  $Y$ be any closed subscheme of $\A^n_\Z$ of codimension $k\geq1$.
Let $r$ and $M$ be positive real numbers.  Then we have
\begin{equation}\label{gseq}
\#\{a\in rB\cap \Z^n\,\,|\,\,a\!\!\!\!\!\pmod{p} \in Y(\F_p)\,\,\mbox{for some prime } p>M\} \,=\, 
O\left(\frac{r^n}{M^{k-1}\log M}+{r^{n-k+1}}\right)
\end{equation}
where the implied constant depends only on $B$ and on $Y$.  
\end{theorem}

\begin{proof}
  We may again assume that $Y$ is irreducible, for otherwise we could
  simply sum over the irreducible components of $Y$.  
Since $Y$ is a closed
  subscheme of $\A^n_\Z$ of codimension $k$, there exist integral
  polynomials $f_1,\ldots,f_k$ such that for any ring $T$, we have
\begin{equation}
Y(T) \subseteq Y''(T):=\{a\in T^n \,|\, f_1(a)=f_2(a) = \cdots = f_k(a) = 0\}
\end{equation}
whose irreducible component $Y'$ containing $Y$ has codimension $k$. 
Furthermore, there exists a constant $M_0>0$ such that  
$Y'$ (mod $p$) has codimension $k$ in $\A^n_{\F_p}$ for all
primes~$p>M_0$.  Since $M_0$ depends only on $Y$, for the purposes of
proving Theorem~\ref{gsthm} we may assume that $M>M_0$.

  Now, by Lemma~\ref{first}, we know that the number of points $a\in
  rB\cap\Z^n$ such that 
  $a\in Y'(\R)$
is $O(r^{n-k})$.
  Thus it      suffices to restrict ourselves to considering those points
  $a\in rB\cap \Z^n$ for which 
  $a\notin Y'(\Z)$. 

Since the result is trivial for $k=1$, we may assume that $k\geq 2$.  
  In this case, we will prove a slight strengthening of the theorem by
  showing
  that 
\begin{equation}\label{strengthening}
\#\{\,(a,p)\,\,|\,\,a\in rB\cap \Z^n, \; p>M,\; a\notin Y'(\Z), \; a\!\!\!\!\!\pmod{p} \in Y'(\F_p)\} \,=\, 
O\left(\frac{r^n}{M^{k-1}\log M}+{r^{n-k+1}}\right).
\end{equation}

  We first count those pairs $(a,p)$ on the left side of
  (\ref{strengthening}) for each prime $p$ satisfying
  $p\leq r$; such primes arise only when $r>M$.  In this case, since
  $\#Y'(\F_p)=O(p^{n-k})$ and $rB$ can be covered by $O((r/p)^n)$
  boxes each of whose sides have length $p$, we conclude that the
  number of $a\in rB\cap\Z^n$ such that $a$ (mod~$p$) is in $Y'(\F_p)$
  is $O(p^{n-k})\cdot O(r^n/p^n) = O(r^n/p^k)$.  Thus the total number
  of desired pairs $(a,p)$ with $p\leq r$ is at most
\begin{equation}\label{indest}
\#\{\,(a,p)\,\,|\,\,a\in rB\cap \Z^n, \; M<p\leq r, \; a\!\!\!\!\!\pmod{p} \in Y'(\F_p)\} \,=\, 
\sum_{M<p\leq r} O\Bigl(\frac{r^n}{p^k}\Bigr) = O\Bigl(\frac{r^n}{M^{k-1}\log M}\Bigr).
\end{equation}

We next turn to the case of counting 
pairs $(a,p)$ 
where $p>r$, and 
show
that 
\begin{equation}\label{indest2}
\#\{\,(a,p)\,\,|\,\,a\in rB\cap \Z^n, \; p>r,\; a\notin Y'(\Z), \;
a\!\!\!\!\!\pmod{p} \in Y'(\F_p)\} \,=\, 
O\left(r^{n-k+1}\right).
\end{equation}
Note first that (\ref{indest2}) is true also for $k=1$: if $a\notin
Y'(\Z)$, then some $f_i$ does not vanish on $a$, and since
$f_i(a)=O(r^{\deg(f_i)})$, we see that $f_i(a)$ can have at most
$O(1)$ prime factors $p>r$. 

For $k\geq 2$, we prove (\ref{indest2}) by induction on $n$. 
As before, we write the $f_i$ as polynomials in the arguments
$x_1,\ldots,x_n$, and for the same reasons as in the proof of
Lemma~\ref{first}, we may assume that: 1) each $f_i$ is irreducible; 2)
$f_1,\ldots,f_{k-1}$ are polynomials only in $x_1,\ldots,x_{n-1}$;
and, 3) the leading coefficient~$h_k$ of $f_k$, as a polynomial in
$x_n$, does not 
vanish on $Y'$. Let $Y_{k-1}$ denote the irreducible component containing $Y'$ of the closed subscheme of $\A^n_\Z$ cut out by $f_1,\ldots,f_{k-1}$, and let~$Z$ denote the union of irreducible components intersecting $Y'\cap\{h_k=0\}$ of the closed subscheme of $\A^n_\Z$ cut out by $f_1,\ldots,f_{k-1},h_k$.  Then $Y_{k-1}$ and $Z$ yield subschemes in $\A^{n}_\Z$ (and in $\A^{n-1}_\Z$, via ignoring the last free coordinate) of codimensions $k-1$ and $k$, respectively.
We may write
\begin{eqnarray}
\label{totalest}
\!\!\!\!\!\!\!\!& \!\!\!\!\!\!\!\!\!\!\!
&\normalsize{\mbox{$\{(a,p)\,|\,a\in rB\cap\Z^n$, $p\!>\!r$, $a\notin
    Y'(\Z)$, $a$ (mod $p$) $\in Y'(\F_p)$\}}}\\
\label{est1}
\!\!\!\!\!\!\!\!&\!\!\!\!\!\!\!\!\!\!\!\normalsize{\subseteq}\!\!\! &
\normalsize{\mbox{$\{(a,p)\,|\,a\in rB\cap\Z^n$, $p\!>\!r$, $a\in Y_{k-1}(\Z)$, $f_k(a)\neq 0$, 
$f_k(a)\equiv$ 0 (mod $p$)\}}\,\cup}\\
\label{est2}
\!\!\!\!\!\!\!\!&\!\!\!\!\!\!\!\!\!\!\! &
\normalsize{\mbox{$\{(a,p)\,|\,a\in rB\cap\Z^n$, $p\!>\!r$, $a\notin
    Y_{k-1}(\Z)$, $a$ (mod $p$) $\in Z(\F_p)$\}}\,\cup} \\
\label{est3}
\!\!\!\!\!\!\!\!&\!\!\!\!\!\!\!\!\!\!\! &
\normalsize{\mbox{$\{(a,p)\,|\,a\in rB\cap\Z^n$, $p\!>\!r$, $a\notin
    Y_{k-1}(\Z)$, $a$ (mod $p$) $\in Y'(\F_p)$, $h_k(a)\not\equiv 0$ (mod $p$)\}}.}
\end{eqnarray}
We estimate the size of the set (\ref{totalest}) by giving estimates
for each of the sets in (\ref{est1})--(\ref{est3}).

Let $P(rB)$ denote the projection of $rB$ onto the first $n-1$
coordinates.  To give an upper estimate for the set in (\ref{est1}),
we note that the number of points $b\in P(rB)\cap \Z^{n-1}$ such that
$(b,\,\cdot\,)\in Y_{k-1}(\Z)$ is 
$O(r^{(n-1)-(k-1)})$ by Lemma~\ref{first}.  If furthermore $(b,a_n)\in
rB\cap\Z^n$ satisfies $f_k(b,a_n)\neq0$, then again $f_k(b,a_n)$ has
at most $O(1)$ prime factors $p>r$.  Thus the total number of
pairs $(a,p)$, where $a=(b,a_n)\in rB\cap \Z^n$ and $p>r$, such that
$a\in Y_{k-1}(\Z)$, 
$f_k(a)\neq 0$, and $f_k(a)\equiv0$ (mod
$p$), is at most $$O(r^{(n-1)-(k-1)})\cdot O(r) \cdot O(1) =
O(r^{n-k+1}),$$ giving the desired estimate for set (\ref{est1}).

Next, we see that the total number of pairs $(a,p)$ in the set
(\ref{est2}) is also at most \linebreak $O(r^{n-k+1})$,
since there are at most $O(r^{(n-1)-k+1)})$ values for the
first $n-1$ arguments by the induction hypothesis applied to $Z$, and then at most
$O(r)$ possible values for the last coordinate for a point in~$rB$.

Finally, we give an upper estimate for the size of the set
(\ref{est3}).  By the induction hypothesis applied to $Y_{k-1}$, the total number of pairs
$(b,p)$, where $b\in P(rB)\cap\Z^{n-1}$ and $p>M$, such that $(b,\,\cdot\,)\notin Y_{k-1}(\Z)$, 
$(b,\,\cdot\,)$ (mod $p$) $\in Y_{k-1}(\F_p)$, 
and $h_k(b,\,\cdot\,)\not\equiv 0$ (mod~$p$) is
$O(r^{(n-1)-(k-1)+1})$.  Given such a pair $(b,p)$, the number of
values of $a_n$ such that $f_k(b,a_n)\equiv 0$ (mod $p$) and
$a=(b,a_n)\in rB\cap\Z^n$ is at most $d\cdot O(1)$, where $d$ denotes
the degree of $f_k$ as a polynomial in $x_n$.  Indeed, $a_n$ (mod $p$)
must be one of the $\leq d$ roots of $f_k$ (mod $p$) in that case, and
the number of such integers $a_n$ in the union of a bounded number of
intervals in $\R$ having total measure at most $O(r)=O(p)$, that are
congruent (mod $p$) to one of these $\leq d$ values, is at most $d\cdot
O(1)=O(1)$ (since $d$ is a constant depending only on $Y$).  We
conclude that the total number of pairs $(a,p)$, where $a=(b,a_n)$,
that lie in the set (\ref{est3}) is at most $O(r^{(n-1)-(k-1)+1})\cdot
O(1) \,=\, O(r^{n-k+1})$, as desired.
\end{proof}
\begin{remark}{\em 
  Tracing through the proof, it is clear  that the bound in Theorem~\ref{gsthm}
can be achieved for suitable choices of $Y$, 
and so
the bound 
is 
 essentially optimal without further assumptions on~$Y$.
    }
\end{remark}

Theorem~\ref{gsthm} also has a number of variations, which can be useful in various sieves depending on context.  One natural variation is when the region of interest in $\R^n$ is not just homogeneously expanding, but is also being applied with more general linear transformations in $\GL_n$, such as diagonal matrices and shears.  

\begin{theorem}\label{gsthm2}
  Let $B$ be a compact region in $\R^n$ having finite measure, and let
  $Y$ be any closed subscheme of $\A^n_\Z$ of codimension $k\geq2$ such that the Zariski closure of the projection of $Y$ onto the first $n-k+j$ coordinates has codimension $j$ in $\A^{n-k+j}$ for $j=0,\ldots,k$. 
Let $r$ and $M$ be positive real numbers, and let $t={\rm
  diag}(t_1,\ldots,t_n)$ be a diagonal element of $\SL_n(\R)$.
Suppose that $\kappa>0$ is a constant such that $rt_i\geq \kappa$ for
all $i$ and $t_i\geq\kappa$ for all $i>n-k$.  Then we have
\begin{equation}\label{gseq2}
\#\{a\in rtB\cap \Z^n\,\,|\,\,a\!\!\!\!\!\pmod{p} \in Y(\F_p)\,\,\mbox{for some prime } p>M\} \,=\, 
O\left(\frac{r^n}{M^{k-1}\log M}+{r^{n-k+1}}\right),
\end{equation}
where the implied constant depends only on $B$, $Y$, and $\kappa$.
\end{theorem}
To prove Theorem~\ref{gsthm2}, we note first that the analogue of
Lemma~\ref{first} holds equally well when $rB$ is replaced with $rtB$,
even without the condition that $t_i\geq \kappa$ for all $i>n-k$.
The proof is then identical to that of Theorem~\ref{gsthm}.

\subsection{Polynomials taking values that are multiples of squares of primes}\label{sqfreeapp}

Let $f$ be a polynomial with integer coefficients in the variables
$x_1,\ldots,x_n$.  To count squarefree values taken by $f$, we wish to
sieve out those points in $\Z^n$ where $f$ is a multiple of $p^2$ for some
prime $p$.  Now if $f(a)\equiv 0$ (mod $p^k$) for some
$k>1$ and $a\in \Z^n$, then this can happen in two distinct ways,
namely, we have either
\begin{equation}\label{strong}
f(a')\equiv 0 \!\!\!\!\pmod{p^k} \quad\!\! \forall a'\equiv a\!\!\!\! \pmod{p}
\end{equation}
or
\begin{equation}\label{weak}
\exists a'\equiv a\!\!\!\! \pmod{p} \mbox{ such that } f(a')\not\equiv 0 \!\!\!\!\pmod{p^k}.
\end{equation}
In the first case, we say that {\it $f$ is strongly a multiple of
  $p^k$ at $a$}, and otherwise we say that {\it $f$ is weakly a
  multiple of $p^k$ at $a$}.  In other words, (\ref{strong}) says that $f(a)$ is
a multiple of $p^k$ for ``mod $p$ reasons''; meanwhile, (\ref{weak}) says
that it is a multiple of $p^k$ for ``mod $p^j$ reasons'', where $j\in \{2,3,\ldots,k\}$ is the smallest integer such that $f(a')\equiv 0$ (mod $p^k$) for all $a'\equiv a$ (mod $p^j$).

These two scenarios are quite different, and it is natural to treat
them separately.  In the case of weak multiples, we will find that,
particularly in cases of high symmetry, one can sometimes prove the
necessary estimates by ring-theoretic methods, or by reducing weak
multiples to the case of strong multiples via more linear algebraic
methods (see \S4 and \S5).
Meanwhile, the study of strong multiples is amenable to 
geometric techniques.  For example, if 
$f$ is strongly a multiple of~$p^k$ at $a=(a_1,\ldots,a_n)\in\Z^n$, and not all the coefficients of $f$ vanish at $a$ (as a polynomial in $x_n$) modulo $p$, then $f(a_1,\ldots,a_{n-1},x_n)$ (mod $p$) must have a root of multiplicity $k$ at $x_n\equiv a_n$ (mod $p$).  
It follows that
if 
$Y_k$ denotes the closed subscheme of $\A^n_\Z$ defined by 
\begin{equation}\label{ykdef}
f=\frac{\partial f}{\partial x_n}=\cdots=\frac{\partial^{k-1}f}{{\partial x_n}^{k-1}}=0,
\end{equation}
then we have for all primes $p$ that
\[ \{a\in\Z^n \,|\, f \mbox{ is strongly a multiple of $p^k$ at $a$}\} \subseteq \{ a\in \Z^n\,|\, a\!\!\!\!\pmod{p} \in Y_k(\F_p)\}.\]
Theorem~\ref{gsthm} can thus be applied in order to estimate the
asymptotic number of points in $\Z^n$, in a homogeneously expanding region,
on which $f$ is strongly a multiple of $p^k$.

Generically, if the degree of $f$ is large enough, then the subscheme
$Y_k$ will have codimension~$k$.  In practice, this can be checked in
any given example; for our purposes, the following lemma will suffice:

\begin{lemma}\label{stronglemma}
Let $f$ be an irreducible integral polynomial in $n\geq 2$ variables.  Then there exists a subscheme $Y$ of $\A^n_\Z$ of codimension two such that, for all primes $p$, we have
\[ \{a\in\Z^n \,|\, f \mbox{ \rm{is strongly a multiple of $p^2$ at $a$}}\} \subseteq \{ a\in \Z^n\,|\, a\!\!\!\!\pmod{p} \in Y(\F_p)\}.\]
\end{lemma}

\begin{proof}
Without loss of generality, we may assume that $f$ is nonconstant as a polynomial in $x_n$.  We let $Y=Y_k$ as defined in (\ref{ykdef}), with $k=2$.  Then since $f$ is irreducible, 
$f$ and $\partial f/\partial x_n$ do not share a common factor, and so they cut out a subscheme in $\A^n_\Z$ of codimension two.  It follows that $Y=Y_k$ as defined in (\ref{ykdef}), with $k=2$, has codimension two in $\A^n_\Z$, as desired.
\end{proof}

Thus, to sieve out lattice points in homogeneously expanding regions
in $\R^n$ where a multivariate irreducible polynomial is a multiple of
$p^2$, one may first use the estimates of \S3.2 to handle the strong
multiples of $p^2$.  It remains to handle the weak multiples; the key
idea then is to utilize extra structure on the polynomials to reduce
weak multiples to strong multiples via appropriate {\em rational} changes of
variable!

\subsection{Proof of Theorem~\ref{maingeneral}}

In this subsection, we prove Theorem~\ref{maingeneral},
i.e., if $f$ is a polynomial that satisfies 
the axioms of Section~2,
then $f$ takes the expected density of squarefree values. 

To this end, 
suppose that $f$ (with given $G$ and $V$) satisfies the set of axioms
of Section~2. We begin by proving Equation (\ref{ld2}) of Remark~\ref{fxversion} for $f$. Let $\FF_1$ and $\FF_X=X^{1/d}\FF_1$ be as in Condition~4. Then Condition 5, in the special case $S=V(\Z)$, states that
\begin{equation}\label{vol}
\left|\FF_X\cap V(\Z)^\gen\right|=\Vol(\FF_X)+o(X^{m/d}).
\end{equation}
For any small $\epsilon>0$, let $\FF_1^{1-\epsilon}$ denote a compact measurable subset of $\FF_1$ such that 
\[ \Vol(\FF_1^{1-\epsilon})=(1-\epsilon)\Vol(\FF_1).\]
(That is, $\FF_1^{1-\epsilon}$ is obtained from $\FF_1$ by cutting off the cusps of $\FF_1$ sufficiently far out.)  
Let $\FF_X^{1-\epsilon}=X^{1/d}\cdot\FF_1^{1-\epsilon}$, so that $$\Vol(\FF_X^{1-\epsilon})=(1-\epsilon)\Vol(\FF_X).$$  Then 
\begin{equation}\label{vol2} \left|\FF_X^{1-\epsilon}\cap
  V(\Z)^\gen\right|=\Vol(\FF_X^{1-\epsilon}) + o(X^{m/d}) =
(1-\epsilon)\Vol(\FF_1)\cdot X^{m/d} + o(X^{m/d}),\end{equation} since we are simply counting
lattice points in a bounded homogeneously expanding region, and then
subtracting away the count of non-generic points which have density zero by Condition~1.
Similarly, for
a set $\mathcal S\subset V(\Z)$ defined by finitely many congruence
conditions, we have
\begin{equation}\label{ramanujan}
 \left|\FF_X^{1-\epsilon}\cap \mathcal S^\gen\right|=
(1-\epsilon)\Vol(\FF_1)\prod_p\mu_p(\mathcal S)\cdot X^{m/d} + o(X^{m/d}),
 \end{equation}
where $\mathcal S^\gen$ denotes the subset of generic points in $\mathcal S$.  Note that, by (\ref{vol}) and (\ref{vol2}), we have
\begin{equation}\label{esmall}
\left|(\FF_X\setminus\FF_X^{1-\epsilon})\cap \mathcal S^\gen\right| \leq \epsilon\cdot\Vol(\FF_1)\cdot X^{m/d} + o(X^{m/d}).
\end{equation}

Now, for each prime $p$, let $S_p$ be a subset of
 $V(\Z)$ 
 defined by finitely many congruence conditions such that for sufficiently large $p$, the set $S_p$ contains all
 elements $v\in V(\Z)$ such that $p^2\nmid f(v)$. Let $S=\cap_p
 S_p$.  Then, to prove (\ref{ld2}) for $f$, it suffices  to determine, asymptotically, the cardinality of $\FF_X\cap S$; indeed, the special case where $S_p$ is exactly the set of elements $v\in V(\Z)$ such that $p^2\nmid f(v)$ will correspond to (\ref{ld2}).
 
Let $M$ be any positive integer.  It follows from Condition~5 that
\begin{equation}\label{finapprox}
\lim_{X\rightarrow\infty} \frac{\left|\FF_X
\cap(\cap_{p\leq M}S_p^\gen)\right|}{X^{m/d}}
= 
\Vol(\FF_1)\prod_{p\leq M}\mu_p(S).
\end{equation}
Letting $M$ tend to $\infty$, we conclude that
\begin{equation}\label{upbound}
\displaystyle{\limsup_{X\rightarrow\infty} \frac{\left|\FF_X
\cap S^\gen\right|}{X^{m/d}}
  \leq  
  \Vol(\FF_1)\prod_{p}\mu_p(S).}
\end{equation}
To obtain a lower bound for $\left|\FF_X
\cap S^\gen \right|$, we note that
\begin{equation}\label{inc}
\bigcap_{p\leq M} S_p \subset 
(S \cup \bigcup_{p>M} W_p),
\end{equation}
where $W_p$ denotes the set of points in $V(\Z)$ having discriminant a multiple of $p^2$.  
We use
 the geometric sieve estimates of the previous section to estimate the size of 
$\FF_X
\cap(\cup_{p> M}W_p^\gen)$.  More precisely, we prove:

\begin{lemma}\label{westimate}
We have $$\left|\FF_X
\cap(\cup_{p> M}W_p^\gen)\right| \,=\, O_\epsilon\bigl(X^{\textstyle{\frac m
  d}}/(M^{\min\{\eta,1\}}\log
M)+X^{\textstyle\frac {m-1} d}\bigr)
+O\bigl(\epsilon X^{\textstyle \frac{m}{d}}\bigr),$$
where the implied constants
are independent of 
 $M$.
\end{lemma}

\begin{proof}
  We write $W_p^\gen = W_p^{(1)}\cup W_p^{(2)}$, where $W_p^{(1)}$ denotes
  the set of points where the discriminant is strongly a multiple of
  $p^2$, and $W_p^{(2)}$ denotes the set of points where the
  discriminant is weakly a multiple of $p^2$.  
  
By Lemma~\ref{stronglemma}, there
exists an arithmetic subscheme $Y$ of
$\A^{m}_\Z=V_\Z$ of codimension $\geq2$ such that $x\in W_p^{(1)}$ implies that $x$
(mod $p$) is a point on $Y(\F_p)$.  Since $\FF_X^{1-\epsilon}$ is a bounded and 
homogeneously expanding region, by (\ref{esmall}) and Theorem~\ref{gsthm} we
conclude that
\begin{equation}\label{w2est}
\begin{array}{rcl}
\left|\FF_X\cap(\cup_{p> M}W_p^{(1)})\right|&=&\left|\FF_X^{1-\epsilon}\cap(\cup_{p> M}W_p^{(1)})\right| +O(\epsilon X^{m/d})\\[.1in] &=& O_\epsilon(X^{m/d}/(M\log M)+X^{(m-1)/d})+O(\epsilon X^{m/d}).
\end{array}
\end{equation}
In particular, any $v\in W_p^{(1)}$ satisfies Condition~6 with $g=1$,
$a=0$, and $k=2$. 

To handle $W_p^{(2)}$ for primes $p$ with $M<p\leq X^{1/(2d)}$, we
may use the same argument used to prove (\ref{indest}) to obtain
\begin{equation}\label{indestwp22}
\#\{\,(v,p)\,\,|\,\,v\in \FF_X^{1-\epsilon}\cap W_p^{(2)}, \; M<p\leq X^{1/(2d)}\} \,=\, 
\sum_{M<p\leq X^{1/(2d)}} O\Bigl(\frac{X^{m/d}}{p^2}\Bigr) = O\Bigl(\frac{X^{m/d}}{M\log M}\Bigr).
\end{equation}
For primes $p>X^{1/(2d)}$, we use Condition 6.   
Let us write $W_p^{(2)}=\cup_{k\geq 0}W_p^{(2)}(k)$, where
$W_p^{(2)}(k)$ is the portion of $W_p$ having given value of $k$ in
Condition~6(ii).  For this fixed $k$, let $\alpha$ be the infimum of $a$ over all $v\in
W_p^{(2)}(k)$. Then we have
\begin{equation}\label{wp2}
N(W_p^{(2)}(k);X)=O( N(V(\Z);X/p^\alpha))=O((X/p^\alpha)^{m/d}),
\end{equation}
where the first equality follows from Conditions 6(i) and 6(iii), and
the fact that $f(v)$ (for $v\in \FF_X\cap V(\Z)$) has at most $d$ prime
factors $p$ greater than $X^{1/(2d)}$ such that $p^2\mid f(v)$; and 
the second equality follows
from Condition~5.  By summing over $p>M'=\max\{M,X^{1/(2d)}\}$, 
this is sufficient to obtain the estimate of
Lemma~\ref{westimate} in cases where $k=0$.  
If $k\geq 1$, then
we may strengthen (\ref{wp2}), when counting in the union of the
$W_p^{(2)}(k)$ over all $p>M'$, using Condition~6(ii), Estimate
(\ref{esmall}), and Theorem~\ref{gsthm}:
\begin{equation}\label{wp22}
\begin{array}{rcl}
\!\!\!N(\cup_{p> M'} W_p^{(2)}(k);X)&\!\!\!\!=\!\!\!\!&O\bigl( \bigl|\{v\in V(\Z):v\in \FF_{X/p^\alpha} \mbox{ and $v$ (mod $p)\in Y_k(\F_p)$ for some $p> M'$}\}\bigr| \bigr)\\[.025in]
&\!\!\!\!=\!\!\!\!&O\bigl( \bigl|\{v\in\FF_{X/M'^\alpha}^{1-\epsilon}\cap
V(\Z):\mbox{$v$ (mod $p)\in Y_k(\F_p)$ for some $p> M'$}\}\bigr| 
\!+\!\epsilon (X/M'^{\,\alpha})^{\textstyle\frac m d}
\bigr)\\[.05in]
&\!\!\!\!=\!\!\!\!& O_\epsilon\bigl((X/M'^{\,\alpha})^{\textstyle\frac m d}/(M'^{\,k-1}\log
M')
+(X/M'^{\,\alpha})^{{\textstyle\frac{m-k+1}{d}}}\bigr)  + O\bigl(\epsilon (X/M'^{\,\alpha})^{\textstyle\frac m d}\bigr).
\end{array}
\end{equation}
Combining (\ref{w2est}), (\ref{indestwp22}), and (\ref{wp22}), we obtain Lemma~\ref{westimate}.
\end{proof}

By (\ref{finapprox}), (\ref{inc}), and Lemma~\ref{westimate}, we see that
\begin{equation}\label{infinapprox}
\liminf_{X\rightarrow\infty} \frac{\left|\FF_X
\cap S^\gen\right|}{X^{m/d}}
\geq 
\Vol(\FF_1)\prod_{p\leq M}\mu_p(S) - O_\epsilon(1/M^{\min\{\eta,1\}}) - O(\epsilon).
\end{equation}
Letting $M$ tend to infinity, and combining with Condition 6(iv) that $\eta>0$, gives
\begin{equation}\label{approx2}
\left|\FF_X\cap S^\gen\right|
\geq 
\Vol(\FF_1)\prod_{p}\mu_p(S)\cdot X^{m/d} -O(\epsilon)\cdot X^{m/d}.
\end{equation}
Finally, letting $\epsilon$ tend to 0, and combining with (\ref{upbound}), yields
\begin{equation}\label{last}
\left|\FF_X\cap S^\gen\right| = \Vol(\FF_1)\prod_p\mu_p(S)\cdot X^{m/d} + o(X^{m/d}).
\end{equation}
This proves (\ref{ld2}) of Remark~\ref{fxversion} under the assumption
that $f$ satisfies the axioms of Section~2.

To prove also (\ref{localdensity}) for $f$ (i.e.,
Theorem~\ref{maingeneral}), 
let $B_N = [-N,N]^m \subset V(\R)$, and for each prime~$p$, let~$c_p$
denote the number of elements $x\in(\Z/p^2\Z)^m$ satisfying $f(x)=0$
in $\Z/p^2\Z$.  Given any positive integer $M$, let $S_M\subset\Z$
denote the set of all integers that are not multiples of $p^2$ for any
prime $p\leq M$.  Then it is clear that
\begin{equation}\label{finiteprod}
\lim_{N\to\infty}\frac{\#\{x\in\Z^m\cap B_N : f(x)\in S_M\}}
{(2N+1)^m} = \prod_{p\leq M}(1-c_p/p^2),
\end{equation}
since the set of points being counted is a union of finitely many
translates of lattices, all defined by congruence conditions modulo a
single fixed modulus (namely, $\prod_{p\leq M}p^2$).  Letting $M$ tend
to infinity, we see that
\begin{equation}\label{lowerprod}
\limsup_{N\to\infty}\frac{\#\{x\in\Z^m\cap B_N : f(x) \mbox{ squarefree}\}}
{(2N+1)^m} \leq  \prod_p (1-c_p/p^2).
\end{equation}
(This upper bound indeed holds for any polynomial, and we have not yet
used any special property of $f(x)$.)

To obtain a lower bound, we note that
by Condition~1, 
the density of non-generic points in
$B_N$ approaches~0 as $N\to\infty$, and hence such non-generic points 
may be ignored for the purposes of proving Theorem~\ref{maingeneral}.  We now treat
separately the generic points on which $f$ is strongly a multiple
of $p^2$ and on which $f$ is weakly a multiple of $p^2$.  In the case
of the set $B_N\cap W_p^{(1)}$ of points in $B_N$ where $f$ is
strongly a multiple of $p^2$, we immediately have by Lemma~\ref{stronglemma} and
Theorem~\ref{gsthm} that 
\begin{equation}\label{one}
\left|B_{N}\cap(\cup_{p\geq M}W_p^{(1)})\right|
=O(N^m/(M\log M) + N^{m-1})\end{equation} 
for any $M>0$.

In order to obtain an analogous estimate for the set $B_N\cap
W_p^{(2)}$ of points in $B_N$ where $f$ is weakly a multiple of $p^2$,
our strategy is to cover a certain large portion of $B_N$ by
fundamental domains for the action of $G(\Z)$ on $V(\R)$, and then
apply (\ref{wp22}) to each such fundamental domain.  To carry out
this plan, we note that $B_N$ may be covered by a countable union $\cup_{i=1}^\infty\gamma_i\FF_X$ ($\gamma_i\in G(\Z)$) of translates of $\FF_X$, where $X$ is sufficiently large so that $|I(v)|<X$ for all $v\in B_N$; 
since $I$ has degree $d$, we may take $X=cN^d$ for some fixed constant
$c>0$.

Let $B_{N,s}=B_N\cap(\cup_{i=1}^s \gamma_i \FF_X)$.  
Since we have
the estimate of Lemma~\ref{westimate} 
for any $G(\Z)$-translate of $\FF_X$,
and $B_{N,s}$ is a union of $s$ translates of
$\FF_X$, we conclude that
\begin{equation}\label{two}
\begin{array}{rcl}
\left|B_{N,s}\cap(\cup_{p>M}W_p^{(2)})\right|&=&
sO_\epsilon(X^{\textstyle\frac m d}/(M^{\min\{\eta,1\}}\log M)+ X^{\textstyle\frac{m-1}d})
+ sO(\epsilon X^{\textstyle\frac m d})\\[.1in]
&=&sO_\epsilon(\:N^m/(M^{\min\{\eta,1\}}\log
M)+\,N^{m-1})  + sO(\epsilon N^m)
\end{array}
\end{equation} 
where the implied constant is independent
of $s$, $M$, and $N$.  

It follows from (\ref{one}) and (\ref{two})  that
\begin{equation}\label{meps}
\liminf_{N\rightarrow\infty} \frac{\#\{x\in\Z^m\cap
  B_{N,s}:f(x)\mbox{ squarefree}\}}{\Vol(B_{N,s})}
\geq \prod_{p\leq M}(1-c_p/p^2) - sO_\epsilon(1/M^{\min\{\eta,1\}}) - sO(\epsilon).
\end{equation}
Evidently, $\Vol(B_{N,s})$ approaches $\Vol(B_N)=(2N)^m$
from below as $s\to\infty$.  Letting $M$ tend to infinity, and then $\epsilon$ to 0, and finally 
$s$ to $\infty$ in (\ref{meps}) now yields the desired result
\begin{equation*}
\lim_{N\rightarrow\infty} \frac{\#\{x\in\Z^m\cap
  B_{N}:f(x)\mbox{ squarefree}\}}{(2N+1)^m}
= \prod_{p}(1-c_p/p^2),
\end{equation*}
since $\Vol(B_N)/(2N+1)^m\to1$ as $N\to\infty$.

\section{The density of squarefree values taken by $f_3$, $f_4$, and
  $f_5$ (Proofs  of Theorems \ref{sqfdisc}--\ref{unrexts})}

In this section, we show that the polynomials $f_3$, $f_4$, and $f_5$
as in the introduction satisfy all the axioms of Section~2.  It will
therefore 
follow by Theorem~\ref{maingeneral} that these polynomials take the
expected density of squarefree values.  
We will also then deduce 
Theorem~\ref{sqfdisc}, Corollary~\ref{sqfreedisccor},
Theorem~\ref{gensqfree}, 
Theorem~\ref{unrexts}, and Theorem~\ref{fthm}.

We begin by describing the polynomials $f_3$, $f_4$, and
$f_5$ in more detail, and their interpretations in terms of cubic,
quartic, and quintic rings. 

\subsection{The parametrization of rings of small rank by prehomogeneous vector spaces}\label{par}

Let $n=3$, 4, or 5.  For any ring $T$ (commutative, with unit), let $V(T)$ be:
\begin{itemize}
\item[(a)] the space $\Sym_3T^2$ of binary cubic forms with coefficients in $T$, if $n=3$; 
\item[(b)] the space $T^2\otimes \Sym_2T^3$ of pairs of ternary quadratic forms with coefficients in $T$, if $n=4$; or 
\item[(c)] the space $T^4\otimes \wedge^2 T^5$ of quadruples of $5\times 5$ skew-symmetric matrices with entries in $T$, if~$n=5$.  
\end{itemize}
Then the group $G(T)$ naturally acts on $V(T)$, where we set $G(T)=\GL_2(T)$,
$\GL_2(T)\times \SL_3(T)$, or $\GL_4(T)\times \SL_5(T)$ in accordance
with whether $n=3$, 4, or 5, respectively. In the case of $n=3$, we
use the ``twisted action'', i.e., 
an element $\gamma\in\GL_2$ acts on a binary cubic form $x(s,t)\in V$ by $$\gamma\cdot x(s,t)=(\det \gamma)^{-1}x((s,t)\cdot\gamma).$$

For each $n=3$, 4, or 5, there is a natural
invariant polynomial $f_n$ for the action of $G(\Z)$ on $V(\Z)$, called the
{\it discriminant}, which in fact generates the ring of polynomial
invariants.  This discriminant polynomial has degree $4$, $12$, or
$40$ on $V(\Z)$, in accordance with whether $n=3$, $4$, or~$5$.  We say
that an orbit of $G(T)$ on $V(T)$ is {\it nondegenerate} if the
discriminant of any element in that orbit is nonzero.

The nondegenerate orbits of $G(T)$ on $V(T)$ in the case of fields $T$
were classified by Wright and Yukie~\cite{WY}, and were shown to be in
natural correspondence with \'etale degree $n$ extensions of~$T$.  The
orbits of $G(\Z)$ on $V(\Z)$ were classified in \cite{GGS}, \cite{hcl3}, and
\cite{hcl4}, where the following theorem was proved:

\begin{theorem}\label{bij}
  The nondegenerate $G(\Z)$-orbits on $V(\Z)$ are in canonical bijection
  with isomorphism classes of pairs $(R,R')$, where $R$ is a ring of
  rank $n$ and $R'$ is a resolvent ring of $R$.  In this bijection,
  the discriminant of an element $v\in V(\Z)$ equals the discriminant
  of the corresponding ring $R$ of rank~$n$.  Furthermore, every
  isomorphism class of ring $R$ of rank~$n$ occurs in this bijection,
  and every isomorphism class of maximal ring occurs exactly once.
\end{theorem}
Recall that a {\it ring of rank $n$} is a ring $R$ (commutative, with
unit) such that $R$ is free of rank $n$ as a $\Z$-module.  A ring of
rank 2, 3, 4, 5, or 6 is called a {\it quadratic}, {\it cubic}, {\it
  quartic}, {\it quintic}, or {\it sextic} ring, respectively. A
resolvent ring $R'$ of a cubic, quartic, or quintic ring $R$ is a
quadratic, cubic, or sextic ring, respectively, that satisfies certain
properties, whose precise definition will not be needed here (see
\cite{hcl3} and \cite{hcl4} for details).  A ring $R$ of rank $n$ is
said to be {\it maximal} if it is not a proper subring of any other
ring of rank $n$; equivalently, $R$ is maximal if it is the maximal
order in a product of number fields.

Now a ring $R$ of rank $n$ that has squarefree (or fundamental)
discriminant is automatically maximal.  In particular, such a ring
will arise exactly once in the bijection of Theorem~\ref{bij}.  We are
interested, however, only in those maximal rings of rank $n$ that are
actually orders in $S_n$-number fields (rather than, say, in a
nontrivial product of number fields).  We thus say that a point $v\in
V(\Z)$ is {\it generic} if the ring $R$ of rank $n$ associated to it
under the bijection of Theorem~\ref{bij} is an order in a $S_n$-number
field of degree $n$.  With this definition of generic in hand, we may
now use the results of \cite{Davenport2,dodqf,dodpf} to show 
that the polynomials $f_n$ ($n=3$, 4, 5) satisfy all the axioms of
Section~2.

\subsection{Verification of axioms for the polynomials $f_3$, $f_4$,
  and $f_5$}

Let again $n=3$, 4, or 5. 
That Conditions 1--5 of the axioms of Section 2 are satisfied by
$f_n$ can be deduced directly from the works \cite{Davenport2},
\cite{dodqf}, and \cite{dodpf} respectively.  Indeed, Condition~1
follows immediately from Hilbert irreducibility or
\cite[\S4]{Davenport2}, \cite[\S2.4]{dodqf}, or
\cite[\S3.2]{dodpf}.  Condition~2 follows from \cite[Lines~4, 8, and~11]{SK}.
For Condition 3, we simply take
$I=f_n$.  Condition~4 is then \cite[\S2]{Davenport2}, \cite[\S2.1]{dodqf},
or \cite[\S2.1]{dodpf}, while Condition 5 is \cite[\S5]{DH},
\cite[Eqn.~(32)]{dodqf}, or \cite[Thm.~17]{dodpf},
respectively.

It remains to check the crucial Condition 6.  We will prove
Condition~6 for $f_3$, $f_4$, and $f_5$ with $a=2$, $c={n\choose2}$,
and $k=0$.  Condition~6(ii) is then automatically satisfied
(noting that genericity is a $G(\Q)$-invariant condition and that
$k=0$).  Also, since $m=d$ in these cases, Condition 6(iv) is also then
automatically satisfied.

We now verify Conditions 6(i) and (iii).  The proofs of these two
important subconditions are where the ``largeness'' of the symmetry group
$G$ of the polynomial $f_n$ is used.

We begin by noting that a general (i.e., nondegenerate) element of $V({\F_p})$
determines $n$ distinct points in $\P^{n-2}(\bar\F_p)$.  Indeed, when
$n=3$, we have a binary cubic form, which generally has 3 zeros in $\P^1$.
Similarly, when $n=4$, we have a pair of conics in $\P^2$, which
generally intersect in 4 points in $\P^2$.  And when $n=5$, we have
four $5\times 5$ skew-symmetric matrices, the $4\times 4$
sub-Pfaffians of which give quadrics in $\P^3$ that generally intersect
in 5 points in $\P^3$. The discriminant of an element $\bar x\in
V(\F_p)$ vanishes precisely when two or more of these $n$ points come
together, or when $\bar x$ is so degenerate that the variety that is
cut out by $\bar x$ in $\P^{n-2}$ is of dimension greater than zero.

The latter case, where a variety of dimension greater than zero is cut
out by $x\in V(\F_p)$, happens on an algebraic set (defined over
$\Z$) that is of codimension greater than one in $V(\F_p)$.
Similarly, the case where strictly fewer than $n-1$ points cut are out
by $x$ (not including multiplicity) also occurs on a set of
codimension greater than one in $V(\F_p)$.  Indeed, these two sets in
$V(\F_p)$ together comprise the image of the set $W_p^{(1)}$ in $V(\F_p)$.

The image of the set $W_p^{(2)}$ in $V(\F_p)$ consists of
elements $x\in V(\F_p)$ that cut out $n$ points (counting
multiplicity) in $\P^{n-2}(\bar\F_p)$, such that two of those $n$
points are the same and the rest are distinct and different.
Thus we have a description of those points in $V({\Z/p\Z})$ on
which the discriminant polynomial $f_n$ vanishes (mod $p$) that
potentially lift to points in $V(\Z/p^2\Z)$ where $f_n$ is {weakly}
a multiple of $p^2$.

To determine what precisely the set is in $V({\Z/p^2\Z})$ where the
discriminant is weakly a multiple of $p^2$, we consider each $n=3$, 4,
or 5 separately.  If $n=3$, then we see from the above discussion that
the image of $W_p^{(2)}$ in $V(\F_p)$ consists only of binary cubic
forms over $\F_p$ having a double (but not triple) root in $\P^1$.  A
binary cubic form in $V(\F_p)$ having a double root in $\P^1$ is
always $G(\F_p)$-equivalent to one of the form $\bar x(s,t)=\bar
as^3+\bar bs^2t$, where $\bar a,\bar b\in \F_p$ and $\bar b\neq 0$.
If $x(s,t)=as^3+bs^2t+cst^2+dt^3\in V(\Z)$ is a lift of $\bar x$ to
$V(\Z)$, then $c$ and $d$ are multiples of $p$, \,$b$ is prime to $p$,
and we compute that the discriminant $f_3(x)$ of $x$ is given by
\[f_3(x) \equiv -4 b^3d \mbox{ (mod $p^2$)} \]
implying (for $p>2$) that $d$ must then be a multiple of $p^2$ for $x$
to have discriminant that is (weakly) a multiple of $p^2$. 

Given such a form $x(s,t)=as^3+bs^2t+cst^2+dt^3\in V(\Z)$, with $p\nmid
b$, $p\mid c$, and $p^2\mid d$, we may multiply $x$ by $p$ and then
apply $g=\bigl[\begin{smallmatrix} 1 & \\ &
  1/p\end{smallmatrix}\bigr]\in\GL_2(\Q)$ to obtain a form
$x'(s,t)=pas^3+bs^2t+(c/p)st^2+(d/p^2)t^3\in V(\Z)$, and we see that
$f_3(x')=f_3(x)/p^2$.  This proves Condition~6(i) for $f_3$ with $a=2$.

To prove Condition 6(iii), note that
an element
$x'(s,t)=a's^3+b's^2t+c'st^2+d't^3\in V(\Z)$ may be sent to an element
$x\in V(\Z)$ of the above type via the inverse transformation $g^{-1}$
precisely
when $p\mid a'$ but $p\nmid b'$.  It follows that the
$\GL_2(\Z)$-class of $x'$ can lead to at most 3 $\GL_2(\Z)$-classes of
elements $x\in V(\Z)$ in this way (since $x'$ can have at most 3 simple
roots in $\P^1(\F_p)$), yielding Condition~6(iii) with $c=3$.  

The case $n=4$ can be treated similarly.  The image of $W_p^{(2)}$ in
$V(\F_p)$ consists only of pairs $(\bar A,\bar B)$ of ternary
quadratic forms over $\F_p$ that have three common zeroes in
$\P^2(\bar \F_p)$ (i.e., two simple common zeroes, and one common zero
of multiplicity two).  By a transformation in $\SL_3(\F_p)$, we may
assume that that the common zero of multiplicity two in $\P^2$ of such
an element $(\bar A,\bar B)\in V(\F_p)$ is at $[1:0:0]\in
\P^2(\F_p)$.  Furthermore, via a transformation in $\GL_2(\F_p)$ we
may assume that $\bar B$ cuts out a union of two distinct lines in
$\P^2(\F_p)$ that intersect at $[1:0:0]\in \P^2$, and $\bar A$ cuts
out a nonsingular conic passing through the intersection point
$[1:0:0]$ of those two lines, but not tangent to either of these two
lines.

It follows that an element in the image of $W_p^{(2)}$ in $V(\F_p)$ will always be $G({\F_p})$-equivalent to an element $(\bar A,\bar B)\in V(\F_p)$ of the matrix form 
$$\left(\left[\begin{array}{ccc}
\bar a_{11}  & \bar a_{12} & \bar a_{13}\\
\bar a_{12} & \bar a_{22} & \bar a_{23} \\
\bar a_{13} & \bar a_{23} & \bar a_{33} 
\end{array}\right], 
\left[\begin{array}{ccc}
\bar b_{11} & \bar b_{12}  & \bar b_{13} \\
\bar b_{12} & \bar b_{22} & \bar b_{23} \\
\bar b_{13} & \bar b_{23} & \bar b_{33} 
\end{array}\right]\right)
$$ where $\bar a_{11} = \bar b_{11} = \bar b_{12} = \bar b_{13} = 0$ and $\bar b_{22}\bar b_{33} - \bar b_{23}^2 \neq 0$.

Now if $(A,B)\in V(\Z)$ is a lift of $(\bar A, \bar B)$ to $V(\Z)$, where $A$ and $B$ have entries $a_{ij}$ and $b_{ij}$ respectively, then $a_{11}$, $b_{11}$, $b_{12}$, and $b_{13}$ are multiples of $p$ and $b_{22}b_{33}-b_{23}^2$ is prime to $p$.  In that case, we compute the discriminant $f_4((A,B))$ of $(A,B)$ to be
\[f_4((A,B)) = \Disc(\det(As + Bt) \equiv  b_{11} (b_{22}b_{33}-b_{23}^2) C^3\mbox{ (mod $p^2$)}, \]
where $C$ is the coefficient of $s^2t$ in $\det(As + Bt)$; this implies that $b_{11}$ must be a multiple of $p^2$ (and $C$ not a multiple of $p$) for $f$ to be (weakly!) a multiple of $p^2$ at $(A,B)$.  

Given such an element $(A,B)\in V(\Z)$, with $a_{11}\equiv b_{12}
\equiv b_{13}\equiv 0$ (mod $p$), $b_{11}\equiv 0$ (mod $p^2$),
$b_{22}b_{33}-b_{23}^2\not\equiv 0$ (mod $p$), and $C\not\equiv 0$
(mod $p$), we may multiply $A$ by $p$ and then divide the first row
and column of both $A$ and $B$ by $p$---this corresponds to the
application of the transformation 
\begin{equation}\label{f4t}
g=\left(\left[\begin{array}{cc} 1 &
  \\ & 1/p\end{array}\right],\left[\begin{array}{ccc} 1/p & & \\
  & 1 & \\ & & 1\end{array}\right]\right)\in G(\Q).
  \end{equation}  
  Hence we
obtain an element $(A',B')\in V(\Z)$ such that $f_4((A',B'))=f_4((A,B))/p^2$,
yielding Condition~6(i) with $a=2$.

To check Condition 6(iii), we note that
an element $(A',B')\in V(\Z)$, where $A'$ and $B'$ have
entries $a'_{ij}$ and $b'_{ij}$ respectively, may be sent to an
element $(A,B)\in V(\Z)$ as above via the inverse transformation
$g^{-1}$ precisely when $a'_{22}\equiv a'_{23}\equiv a'_{33}\equiv 0$ (mod
$p$), i.e., the ternary quadratic form $A'$ (mod $p$) factors into two rational
linear factors with a distinguished linear factor, namely, $x$ (where
$x$ denotes the first variable of the quadratic form), which vanishes
at two of the four common points of intersection of $(A',B')$ in
$\P^2(\bar \F_p)$. It follows that the $G(\Z)$-class of $(A',B')$ can
lead to at most six $G(\Z)$-classes of elements $x\in V(\Z)$ in this way
(since there are at most $6={4 \choose 2}$ lines in $\P^2(\F_p)$
passing through two of four given points), yielding Condition~6(iii) with $c=6$.

The case $n=5$ is more difficult to treat due to the complexity of the
discriminant polynomial, but in the end the same idea still applies.
The image of $W_p^{(1)}$ in $V(\F_p)$ consists of quadruples $(\bar
A,\bar B,\bar C,\bar D)$ of $5\times 5$ skew-symmetric matrices over
$\F_p$ whose $4\times 4$ sub-Pfaffians have four common zeroes in
$\P^3$ (i.e., three simple common zeroes, and one double common zero).
Recall that these common zeroes in $\P^3(\F_p)$ correspond to the linear
combinations of $\bar A,\bar B,\bar C,\bar D$, up to scaling, that
yield rank 2 skew-symmetric matrices.

By a change of basis in $\GL_4(\F_p)$, we may assume that $\bar A$ is
the rank 2 matrix that corresponds to the common double zero in
$\P^3(\F_p)$ of $(\bar A,\bar B,\bar C,\bar D)\in V(\F_p)$.  By a
change of basis in $\SL_5(\F_p)$, we may assume that $\bar A$ is the
$5\times 5$ matrix having $\pm 1$ in the (1,2)- and (2,1)-entries, and
zeroes elsewhere.  We may then use a further
$\GL_4(\F_p)$-transformation to clear out the (1,2) and (2,1) entries
of $\bar B$, $\bar C$, and $\bar D$.
  
We now claim that, after a suitable
$G(\F_p)$-transformation, $(\bar A,\bar B,\bar C,\bar D)$ can be expressed in the form
\begin{equation}\label{reduct}
\left(
\left[\begin{matrix}0&1&0&0&0\\-1&0&0&0&0\\0&0&0&0&0\\0&0&0&0&0\\
0&0&0&0&0\end{matrix}\right], 
\left[\begin{matrix}0&0&\ast&\ast&\ast\\0&0&\ast&\ast&\ast\\\ast&\ast&0&1&0\\\ast&\ast&-1&0&0\\ \ast&\ast&0&0&0\end{matrix}\right], 
\left[\begin{matrix}0&0&\ast&\ast&\ast\\0&0&\ast&\ast&\ast\\\ast&\ast&0&0&1\\\ast&\ast&0&0&0\\ \ast&\ast&-1&0&0\end{matrix}\right], 
\left[\begin{matrix}0&0&\ast&\ast&\ast\\0&0&\ast&\ast&\ast\\\ast&\ast&0&0&0\\\ast&\ast&0&0&0\\ \ast&\ast&0&0&0\end{matrix}\right]\right),
\end{equation}
where the $\ast$'s denote elements of $\F_p$.  

Indeed, if $\bar A$ corresponds to a double and only multiple common zero of $(\bar
A,\bar B,\bar C,\bar D)$, then the coordinate ring of the scheme cut
out by the $4\times 4$ sub-Pfaffians of $(\bar A, \bar B, \bar C,\bar
D)$ in $\P^3$, as given by (16)--(22) in \cite{hcl4}, is isomorphic to
$\F_p[\alpha_1]/(\alpha_1^2)\oplus K$, where $K$ is an \'etale cubic
$\F_p$-algebra.  In the notation of \cite{hcl4}, this means that we must have 
the equalities 
\begin{equation}\label{sl5conds}
Q(\bar A,M_1)\cdot M_2 \cdot Q(M_3,\bar A) = 0
\end{equation}
for any matrices $M_1$, $M_2$, and $M_3$ 
that are $\F_p$-linear combinations of $\bar B$, $\bar C$, and $\bar D$.  Now if the space spanned by the three
bottom right $3\times 3$ matrices of $\bar B$, $\bar C$, and $\bar D$
were three-dimensional, then by assuming that the bottom right $3\times 3$ submatrices of $\bar B$, $\bar C$, and $\bar D$ are 
$\left[\begin{smallmatrix} 0 & 1 & 0\\ -1 & 0 & 0\\ 0&0&0 \end{smallmatrix}\right]$, 
$\left[\begin{smallmatrix} 0 & 0 & 1\\  0 & 0 & 0\\ -1&0&0 \end{smallmatrix}\right]$, and 
$\left[\begin{smallmatrix} 0 & 0 & 0\\ 0 & 0 & 1\\ 0&-1&0 \end{smallmatrix}\right]$, respectively, we see that the conditions
(\ref{sl5conds}) would not hold since, e.g., $Q(\bar A,\bar B)\cdot \bar C \cdot Q(\bar D,\bar A)\neq 0$.  If this space were one-dimensional, then we see that the discriminant of (any lift to $V(\Z)$ of) $(\bar A,\bar B,\bar C,\bar D)$ would be strongly a multiple of $p^2$.  We conclude that this space must be two-dimensional, and a suitable
$\GL_4$-transformation then transforms $(\bar A,\bar B,\bar C,\bar D)$
into the form (\ref{reduct}). 

We now proceed in a manner similar to the case of $n=4$.  Let
$(A,B,C,D)\in V(\Z)$ be an element that reduces modulo $p$ to
(\ref{reduct}).  Then evaluating the discriminant function $f_5$ on this
element, we see that, for those values of the $\ast$'s in
(\ref{reduct}) where the discriminant is not strongly a multiple of
$p^2$, the discriminant of $(A,B,C,D)$ is a multiple of $p^2$
precisely when the (4,5)-entry of $A$ is a multiple of $p^2$.  In that
case, we can multiply $B$, $C$, and $D$ by $p$, and then divide the
fourth and fifth rows and columns of each of $A,B,C,D$ by $p$; this corresponds to
the transformation \begin{equation}\label{f5t}g=(\diag(1,p,p,p),\diag(1,1,1,1/p,1/p))\in G(\Q).\end{equation}
After applying this transformation, we obtain an element
$(A',B',C',D')\in V(\Z)$ such that \linebreak $f_5((A',B',C',D))=f_5((A,B,C,D))/p^2$,
again giving Condition~6(i) with $a=2$.

Finally, to check Condition 6(iii), note that an element
$(A',B',C',D')\in V(\Z)$ may be sent to such an element $(A,B,C,D)\in
V(\Z)$ via the inverse transformation $g^{-1}$ precisely when the top left
$3\times 3$ submatrices of $B',C',D'$ are all zero.  This means that
the fourth and fifth $4\times 4$ sub-Pfaffians of $wA'+xB'+yC'+zD'$
then factor and are both multiples of $w$; it follows that $w$ cuts
out a plane in $\P^3$ that passes through at least 3 of the 5 points
of intersection (counting multiplicity) in $\P^3$ of the $4\times 4$
sub-Pfaffians of $wA'+xB'+yC'+zD'$.  Hence each $G(\Z)$-class of
$(A',B',C',D')$ can lead to at most ${5\choose 3}=10$ $G(\Z)$-classes
of elements $(A,B,C,D)\in V(\Z)$ in this way, proving
that Condition~6(iii) holds with $c=10$.  This completes the proofs of
all the axioms for $f_3$, $f_4$, and $f_5$.

In summary, given a polynomial $f$ on a vector space $V$ that has
symmetry under the action of an algebraic group $G$ on $V$ (all
defined over $\Z$), our general strategy to extract squarefree values
taken by $f$ is to: a) describe geometrically or algebraically what it
means for a point in $V(\F_p)$ to have vanishing $f$ (mod $p$); b)
ascertain which lifts (mod $p^2$) of such points have strongly
vanishing and which have weakly vanishing $f$ (mod~$p^2$); c) treat
the points where $f$ strongly vanishes~(mod~$p^2$) via the geometric
sieve estimates in \S3.2; and, finally, d) for points $x\in V(\Z)$
where $f$ weakly vanishes~(mod~$p^2$), {effect transformations in
  $g\in G(\Q)$ so that the relevant $($mod~$p^2)$ conditions on~$x\in
  V(\Z)$ are transformed into $($mod~$p)$ conditions on $gx\in V(\Z)$}!
We will see that this strategy also works for discriminants of genus
one models in \S5. 

\begin{remark}\label{rmkf} {\em If the
    arguments of \S3.4 (with Theorem 3.5 used in place of
    Theorem~\ref{gsthm}) are applied to each set $H(u,s,\lambda,X)$,
    in the averaging method employed in \cite[\S2.2]{dodpf} (and its
    analogues in \cite[\S5.3]{BST} and \cite[\S2.2]{dodqf}), then one
    obtains Lemma~\ref{westimate} for $f_3$, $f_4$, and $f_5$ without
    the dependence on $\epsilon$ and without the $O(\epsilon X^{m/d})$
    term.  This may be used, e.g., to obtain power saving error terms
    in~(\ref{ld2}), via the methods
    of~\cite{BBP}, \cite{BST}, and \cite{ST}.}
\end{remark}

\subsection{Proof of Theorems~\ref{sqfdisc} and \ref{gensqfree}}
\label{appsieve}

We now consider the consequences for number fields of small degree having
squarefree discriminant.  Let $n\in\{3,4,5\}$.  
 Let $\Sigma=(\Sigma_\nu)_\nu$ denote a set of local
 specifications for degree $n$ number fields, i.e., $\Sigma_p$ is a
 set of (isomorphism classes of) \'etale degree $n$ extensions of
 $\Q_p$, such that for all $p$ larger than some constant $C$, we have
 that $\Sigma_p$ contains all unramified and simply ramified \'etale
 degree $n$ extensions of $\Q_p$.  

For each prime $p$, let $S_p$ denote the subset of
 points of $V(\Z)$ corresponding to rings $R$ such that $R\otimes\Z_p$
 gives the ring of integers of some \'etale extension in $\Sigma_p$.
 If we set $S=\cap_p S_p$, then the generic points of $S$ are
 exactly the points of $V(\Z)$ corresponding to rings $R$ that
 are the rings of integers in $S_n$-number fields of degree
 $n$ that agree with
 the local specifications of $\Sigma$.
Note also that $S$ then satisfies the hypotheses of \S3.4.  Hence Equation
(\ref{last}) holds for $S$. 
 
It thus remains only to compute $\Vol(\FF_1)\prod_p\mu_p(S)$ for the set
$S\subset V(\Z)$ corresponding to the local specifications
$\Sigma=(\Sigma_\infty,\Sigma_2,\Sigma_3,\ldots)$.  For $n=5$, this
has been carried out in \cite[Pf.\ of Lemma~20]{hcl4}, where it is shown that
\[\Vol(\FF_1)\prod_p\mu_p(S) =  \displaystyle{\Bigl(
\sum_{K\in\Sigma_\infty}\frac12\cdot
\frac1{\#\Aut(K)}\Bigr)\prod_p
\Bigl(\sum_{K\in\Sigma_p}}
\frac{p-1}p\cdot\frac1{\Disc_p(K)}\cdot\frac1{\#\Aut(K)} \Bigr),\]
and the identical arguments there show that the same formula holds also in the cases $n=3$ and $n=4$.  
We have proven Theorem~\ref{gensqfree}.

\begin{remark}{\em
Although we do not carry this out here, using the methods of
\cite{BBP,BST,ST}, it is possible to also estimate the $o(X^{m/d})$ in these various deductions and thus obtain a power-saving error term in~(\ref{last}).}
\end{remark}

Theorem~\ref{sqfdisc} is of course the particular case of Theorem~\ref{gensqfree},
where $\Sigma$ corresponds to the local specifications for fields of
squarefree (resp.\ fundamental) discriminant.  In the squarefree case,
$\Sigma_p$ is the set of all isomorphism classes of \'etale algebras
of dimension $n$ over $\Q_p$ such that $p^2\nmid\Disc(K)$.  To
complete the proof of Theorem~\ref{sqfdisc}, we must
evaluate
$$\sum_{K\in \Sigma_p} \frac{1}{\#\Aut(K)}\cdot\frac{1}{\#\Disc_p(K)}.$$
Since, for $p\neq 2$, $p^2\mid \Disc(K)$ is equivalent to $p$ at most
simply ramifying in $K$, while for $p=2$ it means that $p$ does not
ramify, we obtain by \cite[Prop.\ 2.2]{imrn} that
\begin{equation}\label{pmass}
\sum_{K\in \Sigma_p} \frac{1}{\#\Aut(K)}\cdot\frac{1}{\#\Disc_p(K)} = \left\{\begin{array}{cl} 1+1/p & \mbox{if $p\neq 2$} \\[.05in] 1 & \mbox{if $p=2$}\end{array}\right.\;.
\end{equation}
Let $\Sigma_\infty$ denote the set of all \'etale extensions of $\R$
of degree $n$.  Then \cite[Prop.\ 2.4]{imrn} gives
\begin{equation}\label{infmass}
\sum_{K\in \Sigma_\infty} \frac{1}{\#\Aut(K)} =  \frac{r_2(S_n)}{n!}.
\end{equation}
Combining (\ref{pmass}) and (\ref{infmass}), we obtain
$$\displaystyle{\Bigl(
\sum_{K\in\Sigma_\infty}\frac12\cdot
\frac1{\#\Aut(K)}\Bigr)\prod_p
\Bigl(\sum_{K\in\Sigma_p}}
\frac{p-1}p\cdot\frac1{\Disc_p(K)}\cdot\frac1{\#\Aut(K)} \Bigr) =\frac{r_2(S_n)}{2n!}\cdot\frac12\cdot\prod_{p\neq2}\Bigl(1-\frac1{p^2}\Bigr),$$
yielding Theorem~\ref{sqfdisc}(a).

To obtain Theorem~\ref{sqfdisc}(b), we simply must change the local conditions at
$p=2$ to include also simply ramified extensions in $\Sigma_2$.  This
replaces the value 1 for $p=2$ in (\ref{pmass}) by $1+\frac12$ (again,
by \cite[Prop.~2.2]{imrn}), thus multiplying the final constant in
Theorem~\ref{sqfdisc}(a) by $3/2$.  This proves Theorem~\ref{sqfdisc}(b).

\subsection{Unramified extensions of quadratic fields, and proof of Theorem~\ref{unrexts}}

It is known (see, e.g., \cite{Kondo}) that, for $n\leq
5$, an $S_n$-extension of $\Q$ is unramified over its quadratic
subfield precisely when its associated degree $n$ subfield is at most
simply ramified at all places.  Moreover, in this scenario, the
quadratic field ramifies exactly at the places where the degree $n$
field is simply ramified.  In particular, the quadratic field is real
precisely when the degree $n$ field $K$ is totally real, i.e.,
$K\otimes\R\cong\R^n$, and it is imaginary precisely when $K$
satisfies $K\otimes\R\cong \R^{n-2}\times\C$.

As in the proof of the fundamental discriminant case of
Theorem~\ref{sqfdisc}, let $\Sigma_p$ denote the set of all unramified
or simply ramified \'etale extensions of $\Q_p$; furthermore, let
$\Sigma_\infty=\{\R^n\}$.  Then by the arguments of
Theorem~\ref{sqfdisc}, we have that the number of degree $n$ fields
that are simply ramified at all places and have absolute discriminant
less than $X$ is
 $$\displaystyle{\Bigl(
\sum_{K\in\Sigma_\infty}\frac12\cdot
\frac1{\#\Aut(K)}\Bigr)\prod_p
\Bigl(\sum_{K\in\Sigma_p}}
\frac{p-1}p\cdot\frac1{\Disc_p(K)}\cdot\frac1{\#\Aut(K)} \Bigr)\cdot X + o(X),$$ 
which evaluates to 
\begin{equation}\label{realcount}
\frac{1}{2n!}\prod_p\Bigl(1-\frac{1}{p}\Bigr)\Bigl(1+\frac1p\Bigr)\cdot X + o(X) = \frac{1}{2n!}\cdot\zeta(2)^{-1}.
\end{equation}
On the other hand, the number of real quadratic fields having
discriminant less than $X$ is $\frac12\zeta(2)^{-1}\cdot X+o(X)$.  We
conclude that the average number of unramified $(A_n,S_n)$-extensions
of real quadratic fields, over all real quadratic fields of
discriminant less than $X$, as $X\to\infty$, is
$$\frac{ \,\,\,\displaystyle{\frac{1}{2n!}}\zeta(2)^{-1}\,\,\,}{\displaystyle{\frac{1}{2}\zeta(2)^{-1}}} = \frac{1}{n!},$$
yielding Theorem~\ref{unrexts}(a).  

The proof of Theorem~\ref{unrexts}(b) is similar.  We put instead
$K_\infty=\{\R^{n-2}\times\C\}$; this changes the factor of
$\frac{1}{2n!}$ in (\ref{realcount}) to $\frac{1}{2\cdot2(n-2)!}$.
Since the number of imaginary quadratic fields of absolute
discriminant less than $X$ is again $\frac12\zeta(2)^{-1}\cdot
X+o(X)$, we conclude that the average number of unramified
$(A_n,S_n)$-extensions of imaginary quadratic fields over all
imaginary quadratic fields having absolute discriminant less than $X$,
as $X\to\infty$, is
$$\frac{ \displaystyle{\,\,\,\!\frac{1}{2\cdot2(n-2)!}\zeta(2)^{-1}}\!\,\,\,}{\displaystyle{\frac{1}{2}}\zeta(2)^{-1}} = \frac{1}{2(n-2)!},$$
yielding Theorem~\ref{unrexts}(b).

\begin{remark}{\em 
The same argument can also be used to show that the constants occurring in
Theorem~\ref{unrexts} remain the same even when one averages only over
quadratic fields satisfying any descired local conditions at finitely
many primes.}
\end{remark}

\begin{remark}{\em 
We note that analogous results can be proved also for $A_n$-extensions of quadratic fields that are unramified away from some finite set of primes.  For example, if we are interested only in ``weakly unramified extensions'', that is, extensions unramified at all finite places, then the analogous methods would apply; in Theorem~\ref{unrexts}(a)--(b), the constants $1/n!$ and $1/(2(n-2)!)$ would then be replaced by $r_2^+(S_n)/n!$ and $r_2^-(S_n)/n!$, respectively, where $r_2^+(S_n)$ and $r_2^-(S_n)$ denote the number of 2-torsion elements in $S_n$ having signature $+1$ and $-1$, respectively.}
\end{remark}

Next, we may prove by the identical arguments that an unramified
$S_n\times C_2$-extension $M$ of a quadratic field $F$ necessarily
arises as the compositum of the quadratic field $F$ and the Galois
closure $L$ of a number field $K$ of degree $n$ having fundamental
discriminant.  In that case, $\Gal(L/F)=S_n$, and for $M=LF$ to be
unramified over $F$ at all finite places, it is necessary and
sufficient that $\Disc(K)$ divides $\Disc(F)$.  Furthermore, for
$M=LF$ to be unramified over $F$ also at the infinite places, it is
necessary and sufficient that $K$ be totally real if $F$ is real, and
that $K\otimes\R\cong \R^{n-2}\times\C$ if $F$ is imaginary.

Let $N$ be any positive squarefree integer.  Then for $X>N$
sufficiently large, we see that the total number of unramified
extensions of real quadratic fields, where we range over all quadratic
fields of absolute discriminant at most $X$, is $\gg$
\[ \frac{X}{d_1}+\frac{X}{d_2}+\cdots+\frac{X}{d_N},\]
where the $k$-th term above corresponds, by the proof of Theorem~\ref{sqfdisc}(b),
to the count of pairs $(K,F)$, where $K$ is a degree $n$ number field
having fundamental discriminant prime to $d_k$, and
$\Disc(F)=d_k\cdot\Disc(K)$; here $d_k$ ranges over all integer
ratios $\Disc(F)/\Disc(K)$ that are possible for such $(K,F)$ in which
$\Disc(K)\mid \Disc(F)$ and $\Disc(K)\neq \Disc(F)$.  The above sum
$\gg X\log\,X$, proving Theorem~\ref{unrexts}(c).  The argument for
Theorem~\ref{unrexts}(d)---the case of imaginary quadratic fields---is identical.

\section{The density of squarefree values taken by the polynomials
  $g_2$, $g_3$, $g_4$, $g_5$}

Again, most of the axioms of Section 2 follow for the polynomials
$g_2$, $g_3$, $g_4$, $g_5$ of \S1.4, 
using the geometry-of-numbers works \cite{BS,TC,foursel,fivesel}.  In
this section, we outline how to deduce Conditions~1--6 for these
polynomials from the results of these works, with an emphasis again on
Condition~6. 

First, we define a genus one model of degree 2 over $\Q$ (or a
corresponding element of $V(\Q)$) to be {\it generic} if none of the
four ramification points (viewed as a double cover of $\P^1$) is rational
(i.e., the corresponding binary quartic form has no linear factor over
$\Q$).  Similarly, we define a genus one model of degree 3, 4, or 5
over $\Q$ 
(or a corresponding element of $V(\Q)$) to be {\it generic} if it does
not have a rational hyperosculation point (e.g., in the case of
degree 3: does not have a rational flex point).  Alternatively, a genus
one model of degree $n=2$, 3, 4, or 5 over $\Q$ is generic if it does
not correspond to the trivial element of the $n$-Selmer group of its
Jacobian.

Next, we use the following groups $G$ of symmetries of $g_n$
(as a polynomial on $V$) for $n\in\{2,3,4,5\}$:

\begin{itemize}
\item[$n=2:$\!\!\!\!\!\!\!\!]\quad\,
$G=\PGL_2$.  Note that $\gamma\in\GL_2$ naturally acts on a binary quartic form
$x(s,t)\in V$ by $$\gamma\cdot x(s,t)=(\det \gamma)^{-2}x((s,t)\cdot\gamma)$$
yielding an action of $\PGL_2$ on $V$.

\item[$n=3:$\!\!\!\!\!\!\!\!]\quad\,
$G=\PGL_3$. In this case, $\gamma\in\GL_3$ naturally acts on a ternary cubic form
$x(r,s,t)\in V$ by $$\gamma\cdot x(r,s,t)=(\det \gamma)^{-1}x((r,s,t)\cdot\gamma)$$
inducing an action of $\PGL_3$ on $V$.

\item[$n=4:$\!\!\!\!\!\!\!\!]\quad\,
$G=\{(\gamma_2,\gamma_4)\in\GL_2\times\GL_4:\det(\gamma_2)\det(\gamma_4)=1\}/
\{(\lambda^{-2}I_2,\lambda I_4)\},$
where $I_2$ and $I_4$ denote the identity elements of $\GL_2$ and $\GL_4$, and
$\lambda\in\mathbb G_m$.

\item[$n=5:$\!\!\!\!\!\!\!\!]\quad\,
$G=\{(\gamma_1,\gamma_2)\in\GL_5\times\GL_5:\det(\gamma_1)^2\det(\gamma_2)=1\}/\{(\lambda
I_5,\lambda^{-2} I_5)\},$
where $I_5$ denotes the identity element of $\GL_5$ and
$\lambda\in\mathbb G_m$.
\end{itemize}
For $n\in\{2,3,4,5\}$, one easily checks that $g_n$ is an invariant polynomial
for the action of $G$ on $V$.

With these definitions of the groups $G$ and the notion of generic in
hand, Condition~1 then follows again from Hilbert irreducibility or
\cite[\S2.2]{BS}, \cite[\S2.5]{TC}, \cite[\S3.5]{foursel}, and
\cite[3.6]{fivesel}, and Condition~2 from \cite[Thm.~3.2]{BS},
\cite[Prop.~28]{TC}, \cite[Prop.~7]{foursel}, and \cite[Thm.~7]{fivesel}.
For Condition~3, we take $I$ to be the height
$H=\max\{|c_4|^3,|c_6|^2\}$ on $V(\C)$, where $c_4$ and $c_6$ denote the
two generating invariants for the action of $G(\C)$ on $V(\C)$ (see
\cite{Fisher1} for constructions of these invariants).
Condition~4 is then obtained in \cite[\S2.1]{BS}, \cite[\S2.1]{TC},
\cite[\S3.1]{foursel}, and \cite[\S3.1]{fivesel}, while Condition 5 is
\cite[Thm.~2.11]{BS}, \cite[Thm.~17]{TC}, \cite[Thm.~19]{foursel}, and
\cite[Thm.~25]{fivesel}.

Finally, Conditions 6(i)--(iii) for $n=3$, $4$, and $5$ is contained
in \cite[Lem.~26]{TC}, \cite[Lem.~24]{foursel}, and
\cite[Lem.~28]{fivesel}, respectively, with $a=0$, $c=3$, and $k=2$.
Thus Subcondition 6(iv) is then automatically satisfied.  The proofs
of these three important subconditions 6(i)--(iii) are again where the
``largeness'' of the symmetry group $G$ of the polynomial $g_n$ is
used.  We describe now the special case of plane
cubics ($n=3$) in detail to illustrate.

First, recall that for any $n\in\{3,4,5\}$, a general element of
$V({\F_p})$ (i.e., an element for which the discriminant is nonzero)
determines a smooth genus one curve in $\P^{n-1}$ over $\F_p$.  The
discriminant of an element $\bar x\in V(\F_p)$ vanishes precisely when
the associated curve in $\P^{n-1}$ is not smooth, or when $\bar x$ is
so degenerate that the variety that is cut out by $\bar x$ in
$\P^{n-1}$ is of dimension greater than one.

The latter case, where a variety of dimension greater than one is cut
out by $x\in V(\F_p)$, happens on an algebraic set (defined over $\Z$)
that is of codimension greater than one in $V(\F_p)$.  Similarly, the
case where the associated curve in $\P^{n-1}$ has a cuspidal
singularity also occurs on a set of codimension greater than one in
$V(\F_p)$.  Indeed, these two sets in $V(\F_p)$ together comprise the
image of the set $W_p^{(1)}$ in $V(\F_p)$.

The image of the set $W_p^{(2)}$ in $V(\F_p)$ then consists of
elements $x\in V(\F_p)$ that cut out a genus one curve in $\P^{n-1}$ with a
nodal singularity. 
Thus we have a description of those points in $V({\Z/p\Z})$ on
which the discriminant polynomial $g_n$ vanishes (mod $p$) that
potentially lift to points in $V(\Z/p^2\Z)$ on which $g_n$ is {weakly}
a multiple of $p^2$.  

We may now determine precisely the set in $V({\Z/p^2\Z})$ where the
discriminant is weakly a multiple of $p^2$.
When $n=3$, then via a transformation in $G(\Z)$, we may assume that the plane cubic curve over $\F_p$ corresponding to a given element $x\in W_p^{(2)}$ has a node at $(0:0:1)\in\P^2(\F_p)$.  Thus the
corresponding ternary cubic form $\bar x(r,s,t)$ has $t^3$-, $rt^2$-, and
$st^2$-coefficients equal to zero. If $x(r,s,t)$ is a lift of $\bar x$ to
$V(\Z)$, then the $t^3$-, $rt^2$-, and
$st^2$-coefficients of $x(r,s,t)$ are multiples of $p$. 
Evaluating the discriminant of such an element $x$, we see that $g_3(x)\equiv
a_{333}h(x)\pmod{p^2}$, where $a_{333}$ is the coefficient of $z^3$ and $h(x)$ is
an irreducible polynomial in the coefficients of $x$. As $x\in
W_p^{(2)}$, we see that $h(x)\not\equiv0\pmod{p}$. Therefore, since
$p^2\mid g_3(x)$, we obtain that $p^2\mid a_{333}$. Now the element $\gamma$ defined by
\begin{equation}\label{eqmatgamma}
    \left(\begin{smallmatrix}
      1&&\\&1&\\&&p^{-1}
    \end{smallmatrix}\right)\cdot x
\end{equation}
has the same discriminant as $x$ and moreover is in $W_p^{(1)}$, since its $r^3$-, $r^2s$-, $rs^2$-, and $s^3$-coefficients are zero (mod~$p$). We therefore
obtain a discriminant-preserving map $\phi$ from $G(\Z)$-orbits on
$W_p^{(2)}$ to $G(\Z)$-orbits on $W_p^{(1)}$. This proves Conditions~6(i) and (ii) for $g_3$ with $a=0$ and $k=2$.

To prove Condition 6(iii), we note that
an element $x'\in V(\Z)$ may be sent to an element
$x\in V(\Z)$ of the above type via the inverse transformation $\gamma^{-1}$
only when $x'(r,s,t)$ is a multiple of~$t$.  It follows that the
$G(\Z)$-class of $x'$ can lead to at most 3 $G(\Z)$-classes of
elements $x\in V(\Z)$ in this way (since $x'$ can have at most 3 linear factors over $\F_p$), yielding Condition~6(iii) with $c=3$.  

\pagebreak
The cases $n=4$ and $n=5$ can be treated in an analogous fashion; we refer the
reader to \cite[Lem.~24]{foursel} and \cite[Lem.~28]{fivesel},
respectively.  

\begin{remark}\label{rmkg} {\em As in Remark~\ref{rmkf}, if the
    arguments of \S3.4 (with Theorem 3.5 used in place of
    Theorem~\ref{gsthm}) are applied to each set $B(n,t,\lambda,X)$,
    in the averaging method employed in \cite[\S2.3]{BS} (and its
    analogues in \cite[\S2.2]{TC}, \cite[\S3.2]{foursel}, and
    \cite[3.2]{fivesel}), then one again obtains Lemma~\ref{westimate} for
    $g_3$, $g_4$, and $g_5$ without the dependence on $\epsilon$ and
    without the $O(\epsilon X^{m/d})$ term.  As before, this may be
    used, for example, to obtain power saving error terms in~(\ref{ld2}), via the methods of~\cite{BBP},
    \cite{BST}, and \cite{ST}.}
\end{remark}

Condition 6, however, does {\it not} hold for the discriminant polynomial
$g_n$ when $n=2$.  Indeed, in the case $n=2$, the image of 
$W_p^{(2)}$ in $V(\F_p)$ consists of binary quartic
forms over $\F_p$ having exactly one double (but not triple) root in $\P^1$.  A
binary quartic form in $V(\F_p)$ having exactly one double root in $\P^1$ is
always $G(\F_p)$-equivalent to one of the form $\bar x(s,t)=\bar
as^4+\bar bs^3t+\bar cs^2t^2$
where $\bar a,\bar b,\bar c\in \F_p$ and $\bar c\neq 0$ (to prevent a triple root) and $\bar b^2-4\bar a\bar c\neq0$ (to prevent a second double root). If $x(s,t)=as^4+bs^3t+cs^2t^2+dst^3+et^4\in V(\Z)$ is a lift of $\bar x$ to $V(\Z)$, then $d$ and $e$ are multiples of $p$, \,$c$ and $b^2-4ac$ are prime to $p$, 
and we compute that the discriminant $g_2(x)$ of $x$ is given by
\[f(x) \equiv -4c^3(b^2 - 4ac)e \mbox{ (mod $p^2$)} \]
implying (for $p>2$) that $e$ must then be a multiple of $p^2$ for $x$
to have discriminant that is weakly a multiple of $p^2$.  

It is now easy to see that there is {\it no} transformation in $G(\Q)$ that removes the mod $p^2$ condition; in particular, unlike the cases $n=3$, 4, and 5, there is in general no transformation in $G(\Q)$ that maps such an element $x\in W_p^{(2)}$ to $W_p^{(1)}$.  (The best potential candidate is the transformation $\bigl(\begin{smallmatrix} 1 & \\ &p^{-1}\end{smallmatrix}\bigr)$, but this sends $x(s,t)=as^4+bs^3t+cs^2t^2+dst^3+et^4$ to $x'(s,t)=ap^2s^4+bps^3t+cs^2t^2+(d/p)st^3+(e/p^2)t^4$, which in general is again in $W_p^{(2)}$.)  The group $G(\Q)$ of symmetries of $g_2$ is too small to directly apply the methods of Section~3.

To remedy this problem, we use what we call the ``embedding sieve'', where we attempt to embed our orbit space $G(\Z)\backslash V(\Z)$, via some map $\phi$, into another orbit space $G'(\Z)\backslash V'(\Z)$ for which the group $G'(\Z)$ {\it is} sufficiently large.  Moreover, we assume that the invariant polynomial $g$ of interest on $V$ is mapped to a corresponding invariant polynomial $g'$ for the action of $G'$ on $V'$, i.e., for $x\in V(\Z)$, we have $g(x)=g'(\phi(x))$.  If an analogue of Condition 6 then holds for $g'$, $G'$, and $V'$, then the resulting estimate for $W_p^{(2)'}\subset V'(\Z)$ can be pulled back to give an estimate for 
$W_p^{(2)}\subset V(\Z)$, and this may be sufficient to deduce (\ref{localdensity}) for the polynomial $g$.

For example, for the space $V(\Z)$ of binary quartic forms, there are
a number of possibilities for the space $V'$ that yield the desired
estimates for $W_p^{(2)}$. We give here, for simplicity, an example of
$V'$ that we have already treated in the previous section, namely, the
representation $V'$ on which $f_4$ is an invariant polynomial!
Indeed, the group $\PGL_2$ may be viewed as the special orthogonal
group of the three-dimensional quadratic space $W$ of $2\times 2$
matrices of trace zero, with quadratic form $A_1$ given by the
determinant; the representation of $\PGL_2$ on the space $\Sym^4(2)$
of binary quartic forms can than be viewed as the action of $\SO(W)$
by conjugation on the space of self-adjoint operators $T : W \to
W$ with trace zero~\cite{AIT}.  Alternatively, we can view the latter
representation as the action of $\SO(W)$ on pairs $(A,B)$ of quadratic
forms $(A,B)$, where $A=A_1$ and $B$ is the quadratic form given by
$B(w,w)=\langle w,Tw\rangle$.  The space $V'(\Z)$ of pairs $(A,B)$ of
ternary quadratic forms can then be viewed as a representation of the
larger group $G'(\Z)=\GL_2(\Z)\times\SL_3(\Z)$, which has polynomial
invariant $f_4$.  We thus obtain a natural map
\begin{equation}\label{vtow0}
\phi:\PGL_2(\Z)\backslash\Sym^4(\Z^2)=G(\Z)\backslash V(\Z)\to G'(\Z)\backslash V'(\Z)=
\GL_2(\Z)\times\SL_3(\Z)\backslash \Z^2\otimes\Sym^2(\Z^3).
\end{equation}
Explicitly in terms of coordinates, the map $\phi$ is given by
\begin{equation}\label{vtow}
\phi: ax^4+bx^3y+cx^2y^2+dxy^3+ey^4\mapsto \left( \left[ \begin{array}{ccc}\phantom{0} & \phantom{0} & 1/2 \\ \phantom{0} & -1 & \phantom{0} \\ 1/2 & \phantom0 & \phantom0 \end{array} \right],
\left[ \begin{array}{ccc} e & d/2 & 0 \\ d/2 & c & b/2 \\ 0 & b/2 & a \end{array} \right]\right);
\end{equation}
see \cite[\S2.3]{melanie} and \cite[(30)]{BS}.  One easily checks then that for any $v\in V$, we have $g_2(v)=f_4(\phi(v))$.
Furthermore, it was shown in \cite[Prop.~2.16]{BS} that the map $\phi$ defined by (\ref{vtow0}) and (\ref{vtow})
is at most~12-to-1. 

Note that if $x(s,t)=as^4+bs^3t+cs^2t^2+dst^3+et^4$ satisfies $d\equiv 0$ (mod~$p$) and $e\equiv 0$~(mod~$p^2$), then even though there is no transformation in $G(\Q)$ that can be applied to $x$ to remove the mod~$p^2$ condition, there {\it is} a transformation in $G'(\Q)$ that removes the mod~$p^2$ condition on $\phi(x)$, namely, the transformation given by (\ref{f4t})!

We may now proceed as in Section 3.  Let $W_p^{(1)}$ and $W_p^{(2)}$ be the subsets of $V(\Z)$ as defined in \S3.4, and let $W_p^{(1)'}$ and $W_p^{(2)'}$ be their analogues for $V'(\Z)$. The estimate (\ref{w2est}) for $W_p^{(1)}$ is obtained in the identical manner.  Meanwhile, for small primes $p\leq X^{1/6}$, by the argument used to prove the individual estimates (\ref{indest}) for $W_p^{(1)}$ for $p\leq r$ (see also Remark~\ref{rmkf}) gives a useful estimate also for $W_p^{(2)}$; we have
\begin{equation}\label{indestwp2}
|\FF_X\cap W_p^{(2)}| = O(\max\{X^{5/6}/p^2,X^{4/6}\}) \mbox{ for all $p$}.
\end{equation}
We use this estimate for $p\leq X^{1/6}$.  

To handle $p>X^{1/6}$, we use the map $\phi$.  Let $\FF_X'$ be the analogue of $\FF_X$ for $G'$, $V'$, and $f_4$.  Then in the previous section, we have shown that
\begin{equation}\label{indestwpp2}
|\FF_X'\cap W_p^{(2)'}| = O(X/p^2) \mbox{ for all $p$}.
\end{equation}
Since we have a map $\phi:\FF_X\cap V(\Z)\to\FF_X'\cap V'(\Z)$ that is at most 12-to-1 and satisfies $g_2(x)=f_4(\phi(x))$, we conclude using (\ref{indestwp2}) and (\ref{indestwpp2}) that
\begin{equation}\label{indestw22}
\begin{array}{rcl}
|\FF_X\cap (\cup_{p>M} W_p^{(2)}| &=& \displaystyle{
\sum_{M<p\leq X^{1/6}} O(\max\{X^{5/6}/p^2,X^{4/6}\}) 
+\sum_{p>\max\{M,X^{1/6}\}} O( X/p^2)}
\\[.25in]
&=& O(X^{5/6}/\log M).
\end{array}
\end{equation}
The analogue of Lemma~\ref{westimate} for the space $V(\Z)$ of binary quartic forms then becomes
\begin{lemma}\label{westimate2}
We have $$\left|\FF_X
\cap(\cup_{p> M}W_p^\gen)\right|=O_\epsilon(X^{5/6}/\log M+X^{4/6})+O(\epsilon X^{5/6}),$$
where the implied constants
are independent of 
 $M$.
\end{lemma}
As in Remark~\ref{rmkg}, the dependence on $\epsilon$ and the $O(\epsilon X^{5/6})$ term may again be removed if desired by combining with the averaging method of \cite{BS}.

The remainder of the argument in \S3.4 now gives (\ref{ld2}) for these
polynomials $g=g_2$, $g_3$, $g_4$, and $g_5$, which was already used in
\cite{BS,TC,foursel,fivesel} to determine the average orders of 2-, 3-,
4-, and 5-Selmer groups of elliptic curves over $\Q$.  Finally, it also then yields
Theorem~\ref{gthm} giving the density of squarefree values taken by  $g_2$, $g_3$,
$g_4$, and $g_5$.

\subsection*{Acknowledgments}

I am extremely grateful to Jordan Ellenberg, Benedict Gross, Jonathan
Hanke, Wei Ho, Kiran Kedlaya, Juergen Klueners, Hendrik Lenstra,
Henryk Iwaniec, Barry Mazur, Carl Pomerance, Bjorn Poonen, Peter
Sarnak, Arul Shankar, Frank Thorne, and Jerry Wang for many helpful
conversations.  This work was done in part while the author was at
MSRI during the special semester on Arithmetic Statistics.  The author
was also partially supported by NSF grant~DMS-1001828 and a Simons
Investigator Grant.

\end{document}